\documentclass{conm-p-l}

\usepackage{amsmath}
\usepackage{amssymb}
\usepackage{yfonts}
\usepackage{hyperref}
\usepackage{mathrsfs}
\usepackage {latexsym}
\usepackage{graphics}
\usepackage{txfonts}
\usepackage{breqn}
\usepackage{cite}
\usepackage{color}
\usepackage[shortlabels]{enumitem}
\usepackage{wasysym}
\usepackage{tikz}
\usetikzlibrary{snakes}
\usetikzlibrary{arrows.meta}
\usetikzlibrary{decorations.pathmorphing}
\usetikzlibrary{er, positioning} 
\usepackage{mathtools}
\usepackage{comment}
\usepackage{enumitem}
\usepackage{float}

\newif\ifshow

\ifshow

\else
\excludecomment{details}
\fi

\newtheorem{theorem}{Theorem}[section]
\newtheorem{remark}{Remark}[section]
\newtheorem{lemma}[theorem]{Lemma}

\newtheorem{define}{Definition}[section]  
\newtheorem{example}{Example}[section] 

\newtheorem*{proposition*}{Proposition}

\newcommand{\joinR}{\hspace{-.1em}}
\newcommand{\RomanI}{I}
\newcommand{\RomanII}{\mbox{\RomanI\joinR\RomanI}}

\DeclareMathOperator*{\supp}{supp}
\DeclareMathOperator*{\divergence}{div}

\allowdisplaybreaks

\numberwithin{equation}{section}

\begin{document}

\title[Convex integration on singular SPDEs]{Remarks on the convex integration technique applied to singular stochastic partial differential equations}

\author[Hongjie Dong]{Hongjie Dong}
\address{Division of Applied Mathematics, Brown University, 182 George Street, Providence, RI 02912, U.S.A.}

\email{Hongjie_Dong@brown.edu}
\thanks{The first author was partially supported by the NSF under agreement DMS-2055244.}

\author[Kazuo Yamazaki]{Kazuo Yamazaki}  
\address{Department of Mathematics, University of Nebraska, Lincoln, 243 Avery Hall, PO Box, 880130, Lincoln, NE 68588-0130, U.S.A.}
\email{kyamazaki2@nebraska.edu}
\thanks{The second author was partially supported by the NSF Award No. 2531744 and Simons Foundation MPS-TSM-00962572.}

\date{July 1, 2025 and, in revised form, 12/31/2025.}


\begin{abstract}
Singular stochastic partial differential equations informally refer to the partial differential equations with rough random force that leads to the products in the nonlinear terms becoming ill-defined. Besides the theories of regularity structures and paracontrolled distributions, the technique of convex integration has emerged as a possible approach to construct a solution to such singular stochastic partial differential equations. We review recent developments in this area, and also demonstrate that an application of the convex integration technique to prove non-uniqueness seems unlikely for a particular singular stochastic partial differential equation, specifically the $\Phi^{4}$ model from quantum field theory. 
\end{abstract}

\maketitle

\section{Introduction: motivation from physics and real-world applications}\label{Section 1}  
Introduced by Landau and Lifshitz \cite{LL57}, the study of hydrodynamic fluctuations via stochastic partial differential equations - hereafter SPDEs - has since evolved into an important research area. Many physicists followed suit to investigate further: ferromagnet \cite{MM75}; Kardar-Parisi-Zhang (KPZ) equation \cite{KPZ86}; magnetohydrodynamics (MHD) system \cite{CT92}; Navier-Stokes equations \cite{FNS77, YO86}; $\Phi^{4}$ model \cite{PW81, GJ87}; Rayleigh-B$\acute{\mathrm{e}}$nard equation \cite{ACHS81, GP75, HS92, SH77, ZS71}. The common choice of random force among all these aforementioned works and many others that we choose to not mention for brevity is the space-time white noise (STWN) (see \eqref{STWN}). As we will describe in detail, the roughness of the force transmits to the solution and the products within the nonlinear terms become ill-defined, disallowing any way to define a solution in a classical manner; we refer to such SPDEs as singular SPDEs. The celebrated theories of regularity structures \cite{H14} by Hairer and paracontrolled distributions \cite{GIP15} by Gubinelli, Imkeller, and Perkowski now allow such singular SPDEs to admit some solution theory and have led this research area to flourish in the past two decades. Nevertheless, informally stated, even these powerful new methods, without further ideas, seem restricted to solution theory that is only local-in-time, and inapplicable in high dimensions. The technique of convex integration gives a completely new perspective; solutions constructed via this technique are usually always global-in-time, generally considered easier in higher dimensions, and non-unique. In addition to reviewing these developments, we prove in Theorem \ref{Theorem 4.1} that any weak solution to heat equation with damping of odd power with minimum regularity is unique. Its special case is the heat equation with cubic damping, which corresponds to the deterministic $\Phi^{4}$ model (see \eqref{Define Phi4}). This implies that, unless the STWN can stir up non-uniqueness, a successful application of the convex integration technique on the $\Phi^{4}$ model to prove non-uniqueness, which was suggested as an open problem in \cite[Section 4]{Y25e}, is highly unlikely. 

\section{Review on singular SPDEs and convex integration technique} 
We denote $\mathbb{N} \triangleq \{1,2, \hdots, \}$, and $\mathbb{N}_{0} \triangleq \mathbb{N} \cup \{0\}$, and we will mostly work on $\mathbb{T}^{d}$ for $d \in \mathbb{N}$ as the spatial domain of $x$; we write ``$n$D'' for ``$n$-dimensional.'' For brevity we denote $\partial_{t} \triangleq \frac{\partial}{\partial t}$ and $\partial_{tt} \triangleq \frac{\partial^{2}}{(\partial t)^{2}}$. We define $\Lambda^{\gamma} \triangleq (-\Delta)^{\frac{\gamma}{2}}$ as a fractional derivative of order $\gamma \in \mathbb{R}$, specifically a Fourier operator with a symbol $\lvert m \rvert^{\gamma}$ so that $\widehat{ \Lambda^{\gamma} f} (m) = \lvert m \rvert^{\gamma} \hat{f} (m)$.  We recall the H$\ddot{\mathrm{o}}$lder-Besov spaces $\mathscr{C}^{\gamma} \triangleq B_{\infty,\infty}^{\gamma}$ for $\gamma \in \mathbb{R}$ which is equivalent to the classical H$\ddot{\mathrm{o}}$lder spaces $C^{\alpha}$ whenever $\alpha \in (0,\infty) \setminus \mathbb{N}$ although $C^{k} \subsetneq \mathscr{C}^{k}$ for all $k \in \mathbb{N}$ (see \hspace{1sp}\cite[p. 99]{BCD11}). We also recall the Littlewood-Paley decomposition that allows us to write any distribution $f$ as $f = \sum_{j\geq -1} \Delta_{j}f$. Informally, $\Delta_{-1}$ and $\Delta_{j}$ for $j \geq 0$ restrict the frequency support of $f$ near origin and the annulus around $2^{j}$, respectively; we refer to \cite[Section 2.2]{BCD11} for details. Furthermore, Bony's paraproducts allow us to consider a product $fg$ of $f \in \mathscr{C}^{\alpha}$ and $g \in \mathscr{C}^{\beta}$ as 
\begin{equation}\label{Bony}
fg = f \prec g + g \prec f + f \circ g 
\end{equation} 
where
\begin{equation*}
f \prec g  \triangleq \sum_{j\geq -1} \sum_{i=-1}^{j-2} \Delta_{i} f \Delta_{j} g  \hspace{2mm} \text{ and } \hspace{2mm} f \circ g \triangleq \sum_{i, j \geq -1: \lvert i-j \rvert \leq 1 } \Delta_{i} \Delta_{j} g; 
\end{equation*} 
we refer to \cite[Section 2.6]{BCD11} for details. 
\begin{remark}\label{Remark 2.1} 
The product $fg$ in \eqref{Bony}, where $f \in \mathscr{C}^{\alpha}$ and $g \in \mathscr{C}^{\beta}$,  is well-defined if and only if $\alpha + \beta > 0$. This restriction of $\alpha + \beta > 0$ comes only from the resonant term $f \circ g$ (see \cite[p. 12]{GIP15}). 
\end{remark} 

\subsection{Review on singular SPDEs}\label{Section 2.1} 
The purpose of this section is to review recent developments on singular SPDEs. We denote by $\xi$ the STWN, which is a distribution-valued Gaussian field with a covariance of 
\begin{equation}\label{STWN} 
\mathbb{E} [ \xi(t,x) \xi(y,s) ] =  \delta(t-s) \delta(x-y), 
\end{equation} 
where we denoted the Dirac delta at the origin by $\delta$; i.e., assuming for generality that $\xi$ is a vector,  \eqref{STWN} implies 
\begin{align*}
\mathbb{E} [ \xi_{i} (\phi) \xi_{j} (\psi)] = 1_{\{ i = j \}} \int_{\mathbb{R} \times \mathbb{T}^{d}} \phi(t,x) \psi(t,x) dt dx \hspace{3mm} \forall \hspace{1mm} \phi, \psi \in \mathcal{S} ( \mathbb{R} \times \mathbb{T}^{d}),
\end{align*}
where $\mathcal{S}$ denotes the Schwartz space. Informally, white-in-space implies that the covariance consists only of the Dirac delta in space while white-in-time refers to the case of only the Dirac delta in time. 

When a parabolic PDE with heat operator $\partial_{t} - \Delta$ in $\mathbb{T}^{d}$ is forced by STWN $\xi$, we fix a scaling 
\begin{align*}
\mathfrak{s} \triangleq (2, \underbrace{1, \hdots, 1}_{d\text{-many}}).
\end{align*}
Then we associate a corresponding metric and scaling dimension by $d_{\mathfrak{s}}((t,x), (s,y)) \triangleq \lvert t-s \rvert^{\frac{1}{\mathfrak{s}_{0}}} +  \sum_{i=1}^{d} \lvert x_{i} - y_{i} \rvert^{\frac{1}{\mathfrak{s}_{i}}}$ and $\lvert \mathfrak{s} \rvert \triangleq 2 + d$, respectively (see \cite[p. 298]{H14}). We define $\lVert (t,x)- (s,y) \rVert_{\mathfrak{s}} \triangleq d_{\mathfrak{s}}((t,x),(s,y))$. Additionally, we define scaling maps (see \cite[p. 300]{H14}) $\mathscr{S}_{\mathfrak{s}}^{\delta}: \mathbb{R}^{d+1} \mapsto \mathbb{R}^{d+1}$ by 
\begin{equation}
\mathscr{S}_{\mathfrak{s}}^{\delta} ((t,x)) = (\delta^{-\mathfrak{s}_{0}} t, \delta^{-\mathfrak{s}_{1}} x_{1} \hdots, \delta^{-\mathfrak{s}_{d}} x_{d}),  \hspace{4mm} ( \mathscr{S}_{\mathfrak{s}, \bar{z}}^{\delta} \varphi) (z) \triangleq \delta^{-\lvert \mathfrak{s} \rvert} \varphi( \mathscr{S}_{\mathfrak{s}}^{\delta} (z-\bar{z})). 
\end{equation}  
We denote $B_{\mathfrak{s}}(0,1)$ the ball of radius 1 with respect to the distance $d_{\mathfrak{s}}$, centered at the origin. 
\begin{define}
\rm{(\hspace{1sp}\cite[Definition 3.7]{H14})} Let $\alpha < 0$ and $r = - \lceil \alpha \rceil > 0$. We say $\xi \in \mathcal{S}'$ belongs to $\mathscr{C}_{\mathfrak{s}}^{\alpha}$ if it belongs to the dual of $C_{c}^{r}(\mathbb{R} \times \mathbb{T}^{d})$ and for every compact set $K \subset \mathbb{R} \times \mathbb{T}^{d}$ there exists a constant $C$ such that 
\begin{align*}
\langle \xi, \mathscr{S}_{\mathfrak{s},\bar{z}}^{\delta} \eta \rangle  \leq C \delta^{\alpha} \hspace{3mm} \forall \, \eta \in \{ \psi \in C^{r}: \lVert \psi \rVert_{C^{r}} \leq 1, \supp \psi \subset B_{\mathfrak{s}}(0,1)\},  
\end{align*}
for all $\delta \in (0, 1]$ and all $\bar{z} \in K$. 
\end{define}
Under these settings, STWN $\xi$ is known to possess the regularity of 
\begin{equation}\label{Regularity of STWN}
\xi \in \mathscr{C}^{\beta}_{\mathfrak{s}} \hspace{1mm}  \text{ for } \beta < - \frac{d+2}{2} \hspace{3mm} \mathbb{P}\text{-a.s.}
\end{equation} 
according to \cite[Lemma 10.2]{H14} (also \cite[Proposition 3.20]{H14}); we observe that the regularity of $\xi$ deteriorates in higher dimensions.  

To discuss a concrete example, let us consider the Navier-Stokes equations forced by STWN $\xi$: given $u^{\text{in}}$ as the initial data, 
\begin{equation}\label{Navier-Stokes}
\partial_{t} u +  \divergence (u\otimes u) + \nabla \pi = \Delta u + \xi, \hspace{3mm} \nabla\cdot u = 0, 
\end{equation} 
solved by the velocity field $u: \mathbb{R}_{+} \times \mathbb{T}^{d} \mapsto \mathbb{R}^{d}$ and the pressure field $\pi: \mathbb{R}_{+} \times \mathbb{T}^{d} \mapsto \mathbb{R}$ for $d \geq 2$. For simplicity, we assume that $\xi$ is divergence-free. A traditional definition of a solution to SPDE is of mild form (e.g. \cite[Theorem 5.4]{DZ14}) and thus this suggests 
\begin{align}\label{mild Navier-Stokes}
u(t) = e^{t \Delta} u^{\text{in}} - \int_{0}^{t} e^{(t-s) \Delta} \mathbb{P}_{L} \divergence (u\otimes u) (s) ds + \int_{0}^{t} e^{(t-s) \Delta} \xi(s) ds, 
\end{align} 
where $\mathbb{P}_{L}$ denotes the Leray projection onto the space of divergence-free vector fields. Upon taking any norm $\lVert \cdot \rVert$ on both sides so that
\begin{align*}
\lVert u(t) \rVert \leq \lVert e^{t \Delta} u^{\text{in}} \rVert + \left\lVert  \int_{0}^{t} e^{(t-s) \Delta} \mathbb{P}_{L} \divergence (u\otimes u) (s) ds \right\rVert + \left\lVert \int_{0}^{t} e^{(t-s) \Delta} \xi(s) ds \right\rVert, 
\end{align*} 
and considering the well-known property of the heat kernel $e^{(t-s) \Delta}$ (e.g. \cite[Proposition 6.16]{H14}), one deduces that at best, $ \int_{0}^{t} e^{(t-s) \Delta} \xi(s) ds \in \mathscr{C}^{\beta}_{\mathfrak{s}}$ for $\beta < 1- \frac{d}{2}$ $\mathbb{P}$-a.s. and consequently, 
\begin{equation}\label{deduction 1} 
u \in \mathscr{C}^{\beta}_{\mathfrak{s}} \hspace{1mm} \text{ for } \hspace{1mm} \beta< 1 - \frac{d}{2}  \, \mathbb{P}\text{-a.s.} 
\end{equation} 
Then, according to Remark \ref{Remark 2.1}, the product $u \otimes u$ in \eqref{mild Navier-Stokes} is not well-defined for any $d \geq 2$ regardless of the norm $\lVert \cdot \rVert$. 

\begin{remark}\label{Remark 2.2}
The deduction of the regularity of $u$ that we just described is based on the fact that it is a mild solution that satisfies \eqref{mild Navier-Stokes} and it is standard in various classical approaches. For example, Galerkin approximation \cite[Section 8]{CF88} (also mollifier \cite{MB02}, or Fourier truncation \cite{FMRR14}) applied on the deterministic Navier-Stokes equations would require the external force $f$ to be in $L^{2}(0,T; H^{-1}(\mathbb{T}^{d}))$ because the approximate solution satisfies the energy inequality and thus so does the limit. We will see in the Example \ref{Example on singular SQG} that the solution constructed via convex integration can defy this seemingly correct deduction by constructing a solution even when the force is much rougher than expected.  
\end{remark} 

The ill-defined products indicate the necessity to add/subtract renormalization constants upon mollifying the noise in space and taking the limit on the maximal solution. Hairer was the first to realize the potential of the rough path theory due to Lyons \cite{L98} and prove the global solution theory \cite[Theorem 3.6]{H11} for the 1D generalized Burgers' equation forced by STWN, the case in which the solution exists in $\mathscr{C}^{\beta}_{\mathfrak{s}}$ for any $\beta < \frac{1}{2}$ and consequently Young's integral of the nonlinear term becomes ill-defined. In \cite{H13}, Hairer went further and solved the 1D KPZ equation forced by STWN: 
\begin{equation}\label{Define KPZ}
\partial_{t} h = \Delta h + \lvert \nabla h \rvert^{2} + \xi, 
\end{equation} 
solved by $h: \mathbb{R}_{+} \times \mathbb{T} \mapsto \mathbb{R}$ that represents interface height. This singular SPDE can be studied using Cole-Hopf transform and attracted much attention (e.g. \cite{BCJ94, BG97}). Here, $\xi \in \mathscr{C}^{\beta}_{\mathfrak{s}}$ for $\beta < - \frac{3}{2}$ $\mathbb{P}$-a.s. according to \eqref{Regularity of STWN} and hence the analogous deduction to \eqref{deduction 1} leads us to expect $\nabla h \in \mathscr{C}^{\beta}_{\mathfrak{s}}$ for $\beta < -\frac{1}{2}$  $\mathbb{P}$-a.s. making the ``$\lvert \nabla h \rvert^{2}$'' in \eqref{Define KPZ} far from well-defined according to Remark \ref{Remark 2.1}. The rough path theory has proven to be effective, and now there was a need for a systematic approach to construct solution theory for other singular SPDEs. Precisely for this reason, independently, Hairer \cite{H14} and Gubinelli, Imkeller, and Perkowski \cite{GIP15} developed the theories of regularity structures and paracontrolled distributions, respectively. Hairer \cite{GIP15} already treated the local-in-time solution theory for 2D parabolic Anderson model with white-in-space noise and 3D $\Phi^{4}$ model with STWN (also \cite{CC18} by Catellier and Chouk); for a subsequent purpose, we introduce the $\Phi^{4}$ model here:
\begin{equation}\label{Define Phi4}
\partial_{t}u + u^{3} = \Delta u + \xi. 
\end{equation} 
Moreover, Gubinelli, Imkeller, and Perkowski in \cite{GIP15} studied the 1D Burgers' equation forced by STWN $\xi$ and diffused by $-(-\Delta)^{\sigma}$ for $\sigma \in (\frac{5}{6}, 1]$ and 2D generalized parabolic Anderson model. Further applications of these theories to prove local solution theory include the 3D Navier-Stokes equations \cite{ZZ15} and the 3D MHD system \cite{Y23a}. 

There are two caveats worth mentioning concerning these developments. First, all of the solution theories discussed in \cite{H14, GIP15, CC18, ZZ15, Y23a} are limited to local-in-time. There are several special cases in which global solution theory for singular SPDEs have been proven with extra effort, taking advantage of the special structures of the equations. 

\begin{enumerate}[label=(\roman*)]
\item The KPZ equation has an advantage of the Cole-Hopf transform, and Gubinelli and Perkowski observed that the solution to the 1D KPZ equation constructed by Hairer \cite{H13} is global-in-time (see \cite[p. 170 and Corollary 7.5]{GP17}).   
\item The 2D Navier-Stokes equations \eqref{Navier-Stokes} admit an explicit knowledge of its invariant measure $\mu$ due to the divergence-free property that leads to $\int_{\mathbb{T}^{2}} (u\cdot\nabla)u\cdot \Delta u dx = 0$. Exploiting this fact, Da Prato and Debussche \cite{DD02} constructed a global solution theory for the 2D Navier-Stokes equations forced by STWN starting from $\mu$-almost every initial data. More recently, Hairer and Rosati \cite{HR24} provided a new approach to prove such a global solution theory without relying on explicit knowledge of invariant measure; this new approach is based on harmonic analysis estimates and facts from the 2D Anderson Hamiltonian, specifically \cite{AC15} by Allez and Chouk. The approach of \cite{HR24} has some flexibility and was extended to the 2D MHD system forced by STWN which does not admit an explicit knowledge of its invariant measure (\hspace{1sp}\cite{Y23c}), and the 1D Burgers' equation forced by $(-\Delta)^{\frac{1}{4}} \xi$ with $\xi$ being the STWN (\hspace{1sp}\cite{Y25f}). In order for the approach of \cite{HR24} to work, it seems necessary that the renormalization constant grows at most logarithmically and the diffusion $-(-\Delta)^{m}$ must be at least as strong as the energy-critical level which informally requires that $m \geq \frac{2+d}{4}$; we refer to \cite[Section 1.3]{Y25f} on this discussion. More recently, the approach of \cite{HR24} was applied on the 3D Navier-Stokes equations forced by STWN and diffused by $-(-\Delta)^{\frac{5}{4}}$. 
\item The $\Phi^{4}$ model \eqref{Define Phi4} has a nonlinear term that has damping effect, and for this reason, Da Prato and Debussche in \cite[Theorem 4.2]{DD03b} were able to extend their strategy in \cite{DD02} to prove its global solution theory (also \cite{MR99}). Subsequently, Mourrat and Weber \cite{MW17} extended such a result to the whole plane. 
\end{enumerate}  

The second caveat is that broadly stated, both theories of regularity structures and paracontrolled distributions seem limited to the case of locally subcritical SPDEs, the case in which, informally stated, the homogeneity of the nonlinear term exceeds that of the force; the precise statement can be found in \cite[Assumption 8.3]{H14}. Let us describe this criterion in detail through the following example of the Navier-Stokes equations \eqref{Navier-Stokes}. 

\begin{example}\label{Example 2.1}
Recalling that the homogeneity of a product is sum of homogeneities, it follows from \eqref{deduction 1} that $\divergence (u\otimes u)$ has the homogeneity akin to $\mathscr{C}^{\beta}_{\mathfrak{s}}$ for $\beta < 1 - d$. The locally subcritical case requires then that $1 - d > - \frac{d+2}{2}$ according to \eqref{Regularity of STWN}, which is precisely when $4 > d$, establishing that the 4D case is locally critical and any dimension above is locally supercritical, to which even the theories of regularity structures or paracontrolled distributions alone are not expected to attain even a local-in-time solution theory. 

Let us examine the differences among the locally subcritical, critical, and supercritical cases in more detail. As we now know that the locally critical case is $d=4$, we consider $d \in \{3,4\}$. The computations here basically follow those of \cite{ZZ15} (see also \cite{Y23a}). First, we denote the STWN $\xi$ by a black circle $\scalebox{0.18}{\begin{tikzpicture}
\filldraw[black] (0,0) circle (6pt); 
\end{tikzpicture}
}$, convolution against a heat kernel $e^{t \Delta}$ by a straight black line $\scalebox{0.18}{\begin{tikzpicture}
\draw[black, thick] (0,1.5) -- (0,0);
\end{tikzpicture}
}$ , and convolution against a heat kernel $e^{t \Delta}$ after $\mathbb{P}_{L}\divergence$ by a zigzag line 
$\scalebox{0.16}{\begin{tikzpicture}
\draw[snake=zigzag](0,1.5) -- (0,0);
\end{tikzpicture}
}$. 
\begin{enumerate}[label=(\roman*)]
\item We apply $\mathbb{P}_{L}$ to \eqref{Navier-Stokes} and consider the following stochastic Stokes equation:
\begin{equation}\label{Define z1}
\partial_{t}  z^{ 
\scalebox{0.18}{\begin{tikzpicture}
\draw[black, thick] (0.5,1) -- (0.5,0);
\filldraw[black] (0.5,1) circle (6pt); 
\end{tikzpicture}
}} - \Delta z^{ 
\scalebox{0.18}{\begin{tikzpicture}
\draw[black, thick] (0.5,1) -- (0.5,0);
\filldraw[black] (0.5,1) circle (6pt); 
\end{tikzpicture}
}} = \xi, \hspace{3mm}  z^{ 
\scalebox{0.18}{\begin{tikzpicture}
\draw[black, thick] (0.5,1) -- (0.5,0);
\filldraw[black] (0.5,1) circle (6pt); 
\end{tikzpicture}
}} \rvert_{t=0} = 0. 
\end{equation} 
We define $u_{1} \triangleq u - z^{ 
\scalebox{0.18}{\begin{tikzpicture}
\draw[black, thick] (0.5,1) -- (0.5,0);
\filldraw[black] (0.5,1) circle (6pt); 
\end{tikzpicture}
}}$ which then solves 
\begin{equation}\label{Define u1}
\partial_{t} u_{1} - \Delta u_{1} = - \mathbb{P}_{L} \divergence (u_{1} + z^{ 
\scalebox{0.18}{\begin{tikzpicture}
\draw[black, thick] (0.5,1) -- (0.5,0);
\filldraw[black] (0.5,1) circle (6pt); 
\end{tikzpicture}
}} )^{\otimes 2},
\end{equation} 
where we defined $X^{\otimes 2} \triangleq X \otimes X$. An analogous reasoning to \eqref{deduction 1} shows that 
\begin{equation}\label{Regularity of z1}
z^{ 
\scalebox{0.18}{\begin{tikzpicture}
\draw[black, thick] (0.5,1) -- (0.5,0);
\filldraw[black] (0.5,1) circle (6pt); 
\end{tikzpicture}
}} \in \mathscr{C}^{\beta}_{\mathfrak{s}} \text{ for } \beta < 1 - \frac{d}{2}, 
\end{equation} 
yielding $u_{1} \in \mathscr{C}^{\beta}_{\mathfrak{s}}$ for $\beta < 3-d$ $\mathbb{P}$-a.s.
\item As $u_{1}$ is still a distribution, we now consider 
\begin{equation}\label{Define z2}
\partial_{t} z^{ 
\scalebox{0.18}{\begin{tikzpicture}
\draw[black, thick] (-0.7,0.9) -- (0,0);
\draw[black, thick] (0.7,0.9) -- (0,0);
\draw[snake=zigzag] (0,0) -- (0,-1);
\filldraw[black] (-0.7,0.9) circle (6pt); 
\filldraw[black] (0.7,0.9) circle (6pt); 
\end{tikzpicture}
}} - \Delta z^{ 
\scalebox{0.18}{\begin{tikzpicture}
\draw[black, thick] (-0.7,0.9) -- (0,0);
\draw[black, thick] (0.7,0.9) -- (0,0);
\draw[snake=zigzag] (0,0) -- (0,-1);
\filldraw[black] (-0.7,0.9) circle (6pt); 
\filldraw[black] (0.7,0.9) circle (6pt); 
\end{tikzpicture}
}} = - \mathbb{P}_{L} \divergence ( z^{ 
\scalebox{0.18}{\begin{tikzpicture}
\draw[black, thick] (0.5,1) -- (0.5,0);
\filldraw[black] (0.5,1) circle (6pt); 
\end{tikzpicture}
}})^{\otimes 2}, \hspace{3mm} z^{ 
\scalebox{0.18}{\begin{tikzpicture}
\draw[black, thick] (-0.7,0.9) -- (0,0);
\draw[black, thick] (0.7,0.9) -- (0,0);
\draw[snake=zigzag] (0,0) -- (0,-1);
\filldraw[black] (-0.7,0.9) circle (6pt); 
\filldraw[black] (0.7,0.9) circle (6pt); 
\end{tikzpicture}
}} \rvert_{t=0} = 0. 
\end{equation} 
We then define $u_{2} \triangleq u_{1} - z^{ 
\scalebox{0.18}{\begin{tikzpicture}
\draw[black, thick] (-0.7,0.9) -- (0,0);
\draw[black, thick] (0.7,0.9) -- (0,0);
\draw[snake=zigzag] (0,0) -- (0,-1);
\filldraw[black] (-0.7,0.9) circle (6pt); 
\filldraw[black] (0.7,0.9) circle (6pt); 
\end{tikzpicture}
}}$ so that $u_{2}$ satisfies 
\begin{equation}\label{Define u2}
\partial_{t} u_{2} - \Delta u_{2} = - \mathbb{P}_{L} \divergence  \left( ( u_{2} + z^{ 
\scalebox{0.18}{\begin{tikzpicture}
\draw[black, thick] (-0.7,0.9) -- (0,0);
\draw[black, thick] (0.7,0.9) -- (0,0);
\draw[snake=zigzag] (0,0) -- (0,-1);
\filldraw[black] (-0.7,0.9) circle (6pt); 
\filldraw[black] (0.7,0.9) circle (6pt); 
\end{tikzpicture}
}})^{\otimes 2} + ( u_{2} + z^{ 
\scalebox{0.18}{\begin{tikzpicture}
\draw[black, thick] (-0.7,0.9) -- (0,0);
\draw[black, thick] (0.7,0.9) -- (0,0);
\draw[snake=zigzag] (0,0) -- (0,-1);
\filldraw[black] (-0.7,0.9) circle (6pt); 
\filldraw[black] (0.7,0.9) circle (6pt); 
\end{tikzpicture}
}}) \otimes z^{ 
\scalebox{0.18}{\begin{tikzpicture}
\draw[black, thick] (0.5,1) -- (0.5,0);
\filldraw[black] (0.5,1) circle (6pt); 
\end{tikzpicture}
}} + z^{ 
\scalebox{0.18}{\begin{tikzpicture}
\draw[black, thick] (0.5,1) -- (0.5,0);
\filldraw[black] (0.5,1) circle (6pt); 
\end{tikzpicture}
}} \otimes (u_{2} + z^{ 
\scalebox{0.18}{\begin{tikzpicture}
\draw[black, thick] (-0.7,0.9) -- (0,0);
\draw[black, thick] (0.7,0.9) -- (0,0);
\draw[snake=zigzag] (0,0) -- (0,-1);
\filldraw[black] (-0.7,0.9) circle (6pt); 
\filldraw[black] (0.7,0.9) circle (6pt); 
\end{tikzpicture}
}} ) \right). 
\end{equation} 
Considering \eqref{Regularity of z1}, we see that 
\begin{equation}\label{Regularity of z2}
z^{ 
\scalebox{0.18}{\begin{tikzpicture}
\draw[black, thick] (-0.7,0.9) -- (0,0);
\draw[black, thick] (0.7,0.9) -- (0,0);
\draw[snake=zigzag] (0,0) -- (0,-1);
\filldraw[black] (-0.7,0.9) circle (6pt); 
\filldraw[black] (0.7,0.9) circle (6pt); 
\end{tikzpicture}
}} \in \mathscr{C}^{\beta}_{\mathfrak{s}} \text{ for } \beta < 3-d. 
\end{equation}
The most singular term in \eqref{Define u2} is $\divergence (z^{ 
\scalebox{0.18}{\begin{tikzpicture}
\draw[black, thick] (-0.7,0.9) -- (0,0);
\draw[black, thick] (0.7,0.9) -- (0,0);
\draw[snake=zigzag] (0,0) -- (0,-1);
\filldraw[black] (-0.7,0.9) circle (6pt); 
\filldraw[black] (0.7,0.9) circle (6pt); 
\end{tikzpicture}
}} \otimes z^{ 
\scalebox{0.18}{\begin{tikzpicture}
\draw[black, thick] (0.5,1) -- (0.5,0);
\filldraw[black] (0.5,1) circle (6pt); 
\end{tikzpicture}
}})$ of regularity in $\mathscr{C}^{\beta}_{\mathfrak{s}}$ for $\beta < 3 - \frac{3d}{2}$, which allows us to conclude that $u_{2} \in \mathscr{C}^{\beta}_{\mathfrak{s}}$ for $\beta < 5 - \frac{3d}{2}$ $\mathbb{P}$-a.s. This implies that in case $d = 3$, the solution is function-valued; yet, because the products in \eqref{Define u2} are still ill-defined, we go further.
 \item We define 
\begin{equation}\label{Define z3}
\partial_{t} z^{
\scalebox{0.16}{\begin{tikzpicture}
\draw[black, thick] (-0.7,0.9) -- (0,0);
\draw[black, thick] (0.7,0.9) -- (0,0);
\draw[black, thick] (0.7,0) -- (0,-0.9);
\draw[snake=zigzag](0,0) -- (0,-0.9);
\draw[snake=zigzag](0,-0.9) -- (0,-1.8);
\filldraw[black] (-0.7,0.9) circle (7pt); 
\filldraw[black] (0.7,0.9) circle (7pt); 
\filldraw[black] (0.7,0) circle (7pt); 
\end{tikzpicture}
}} - \Delta z^{
\scalebox{0.16}{\begin{tikzpicture}
\draw[black, thick] (-0.7,0.9) -- (0,0);
\draw[black, thick] (0.7,0.9) -- (0,0);
\draw[black, thick] (0.7,0) -- (0,-0.9);
\draw[snake=zigzag](0,0) -- (0,-0.9);
\draw[snake=zigzag](0,-0.9) -- (0,-1.8);
\filldraw[black] (-0.7,0.9) circle (7pt); 
\filldraw[black] (0.7,0.9) circle (7pt); 
\filldraw[black] (0.7,0) circle (7pt); 
\end{tikzpicture}
}} = - \mathbb{P}_{L} \divergence \left( z^{ 
\scalebox{0.18}{\begin{tikzpicture}
\draw[black, thick] (-0.7,0.9) -- (0,0);
\draw[black, thick] (0.7,0.9) -- (0,0);
\draw[snake=zigzag] (0,0) -- (0,-1);
\filldraw[black] (-0.7,0.9) circle (6pt); 
\filldraw[black] (0.7,0.9) circle (6pt); 
\end{tikzpicture}
}} \otimes z^{ 
\scalebox{0.18}{\begin{tikzpicture}
\draw[black, thick] (0.5,1) -- (0.5,0);
\filldraw[black] (0.5,1) circle (6pt); 
\end{tikzpicture}
}} + z^{ 
\scalebox{0.18}{\begin{tikzpicture}
\draw[black, thick] (0.5,1) -- (0.5,0);
\filldraw[black] (0.5,1) circle (6pt); 
\end{tikzpicture}
}} \otimes z^{ 
\scalebox{0.18}{\begin{tikzpicture}
\draw[black, thick] (-0.7,0.9) -- (0,0);
\draw[black, thick] (0.7,0.9) -- (0,0);
\draw[snake=zigzag] (0,0) -- (0,-1);
\filldraw[black] (-0.7,0.9) circle (6pt); 
\filldraw[black] (0.7,0.9) circle (6pt); 
\end{tikzpicture}
}} \right), \hspace{3mm} z^{
\scalebox{0.16}{\begin{tikzpicture}
\draw[black, thick] (-0.7,0.9) -- (0,0);
\draw[black, thick] (0.7,0.9) -- (0,0);
\draw[black, thick] (0.7,0) -- (0,-0.9);
\draw[snake=zigzag](0,0) -- (0,-0.9);
\draw[snake=zigzag](0,-0.9) -- (0,-1.8);
\filldraw[black] (-0.7,0.9) circle (7pt); 
\filldraw[black] (0.7,0.9) circle (7pt); 
\filldraw[black] (0.7,0) circle (7pt); 
\end{tikzpicture}
}} \rvert_{t=0} = 0, 
\end{equation} 
so that $u_{3} \triangleq u_{2} - z^{
\scalebox{0.16}{\begin{tikzpicture}
\draw[black, thick] (-0.7,0.9) -- (0,0);
\draw[black, thick] (0.7,0.9) -- (0,0);
\draw[black, thick] (0.7,0) -- (0,-0.9);
\draw[snake=zigzag](0,0) -- (0,-0.9);
\draw[snake=zigzag](0,-0.9) -- (0,-1.8);
\filldraw[black] (-0.7,0.9) circle (7pt); 
\filldraw[black] (0.7,0.9) circle (7pt); 
\filldraw[black] (0.7,0) circle (7pt); 
\end{tikzpicture}
}}$ satisfies 
\begin{equation}\label{Define u3} 
\partial_{t} u_{3} - \Delta u_{3} = - \mathbb{P}_{L} \divergence \left( ( u_{3} + z^{ 
\scalebox{0.18}{\begin{tikzpicture}
\draw[black, thick] (-0.7,0.9) -- (0,0);
\draw[black, thick] (0.7,0.9) -- (0,0);
\draw[snake=zigzag] (0,0) -- (0,-1);
\filldraw[black] (-0.7,0.9) circle (6pt); 
\filldraw[black] (0.7,0.9) circle (6pt); 
\end{tikzpicture}
}} + z^{
\scalebox{0.16}{\begin{tikzpicture}
\draw[black, thick] (-0.7,0.9) -- (0,0);
\draw[black, thick] (0.7,0.9) -- (0,0);
\draw[black, thick] (0.7,0) -- (0,-0.9);
\draw[snake=zigzag](0,0) -- (0,-0.9);
\draw[snake=zigzag](0,-0.9) -- (0,-1.8);
\filldraw[black] (-0.7,0.9) circle (7pt); 
\filldraw[black] (0.7,0.9) circle (7pt); 
\filldraw[black] (0.7,0) circle (7pt); 
\end{tikzpicture}
}} )^{\otimes 2} + (u_{3} + z^{
\scalebox{0.16}{\begin{tikzpicture}
\draw[black, thick] (-0.7,0.9) -- (0,0);
\draw[black, thick] (0.7,0.9) -- (0,0);
\draw[black, thick] (0.7,0) -- (0,-0.9);
\draw[snake=zigzag](0,0) -- (0,-0.9);
\draw[snake=zigzag](0,-0.9) -- (0,-1.8);
\filldraw[black] (-0.7,0.9) circle (7pt); 
\filldraw[black] (0.7,0.9) circle (7pt); 
\filldraw[black] (0.7,0) circle (7pt); 
\end{tikzpicture}
}} ) \otimes z^{ 
\scalebox{0.18}{\begin{tikzpicture}
\draw[black, thick] (0.5,1) -- (0.5,0);
\filldraw[black] (0.5,1) circle (6pt); 
\end{tikzpicture}
}} + z^{ 
\scalebox{0.18}{\begin{tikzpicture}
\draw[black, thick] (0.5,1) -- (0.5,0);
\filldraw[black] (0.5,1) circle (6pt); 
\end{tikzpicture}
}} \otimes (u_{3} + z^{
\scalebox{0.16}{\begin{tikzpicture}
\draw[black, thick] (-0.7,0.9) -- (0,0);
\draw[black, thick] (0.7,0.9) -- (0,0);
\draw[black, thick] (0.7,0) -- (0,-0.9);
\draw[snake=zigzag](0,0) -- (0,-0.9);
\draw[snake=zigzag](0,-0.9) -- (0,-1.8);
\filldraw[black] (-0.7,0.9) circle (7pt); 
\filldraw[black] (0.7,0.9) circle (7pt); 
\filldraw[black] (0.7,0) circle (7pt); 
\end{tikzpicture}
}} ) \right). 
\end{equation} 
From the previous observation that $\divergence (z^{ 
\scalebox{0.18}{\begin{tikzpicture}
\draw[black, thick] (-0.7,0.9) -- (0,0);
\draw[black, thick] (0.7,0.9) -- (0,0);
\draw[snake=zigzag] (0,0) -- (0,-1);
\filldraw[black] (-0.7,0.9) circle (6pt); 
\filldraw[black] (0.7,0.9) circle (6pt); 
\end{tikzpicture}
}} \otimes z^{ 
\scalebox{0.18}{\begin{tikzpicture}
\draw[black, thick] (0.5,1) -- (0.5,0);
\filldraw[black] (0.5,1) circle (6pt); 
\end{tikzpicture}
}})$ has regularity of $\mathscr{C}^{\beta}_{\mathfrak{s}}$ for $\beta < 3 - \frac{3d}{2}$, we deduce that 
\begin{equation}\label{Regularity of z3}  
z^{
\scalebox{0.16}{\begin{tikzpicture}
\draw[black, thick] (-0.7,0.9) -- (0,0);
\draw[black, thick] (0.7,0.9) -- (0,0);
\draw[black, thick] (0.7,0) -- (0,-0.9);
\draw[snake=zigzag](0,0) -- (0,-0.9);
\draw[snake=zigzag](0,-0.9) -- (0,-1.8);
\filldraw[black] (-0.7,0.9) circle (7pt); 
\filldraw[black] (0.7,0.9) circle (7pt); 
\filldraw[black] (0.7,0) circle (7pt); 
\end{tikzpicture}
}} \in \mathscr{C}^{\beta}_{\mathfrak{s}} \text{ for } \beta < 5 - \frac{3d}{2}. 
\end{equation} 
The singular products of \eqref{Define u3} are particularly $\divergence (z^{ 
\scalebox{0.18}{\begin{tikzpicture}
\draw[black, thick] (-0.7,0.9) -- (0,0);
\draw[black, thick] (0.7,0.9) -- (0,0);
\draw[snake=zigzag] (0,0) -- (0,-1);
\filldraw[black] (-0.7,0.9) circle (6pt); 
\filldraw[black] (0.7,0.9) circle (6pt); 
\end{tikzpicture}
}})^{\otimes 2}$ and $\divergence ( z^{
\scalebox{0.16}{\begin{tikzpicture}
\draw[black, thick] (-0.7,0.9) -- (0,0);
\draw[black, thick] (0.7,0.9) -- (0,0);
\draw[black, thick] (0.7,0) -- (0,-0.9);
\draw[snake=zigzag](0,0) -- (0,-0.9);
\draw[snake=zigzag](0,-0.9) -- (0,-1.8);
\filldraw[black] (-0.7,0.9) circle (7pt); 
\filldraw[black] (0.7,0.9) circle (7pt); 
\filldraw[black] (0.7,0) circle (7pt); 
\end{tikzpicture}
}} \otimes z^{ 
\scalebox{0.18}{\begin{tikzpicture}
\draw[black, thick] (0.5,1) -- (0.5,0);
\filldraw[black] (0.5,1) circle (6pt); 
\end{tikzpicture}
}})$, both of which belong to $\mathscr{C}^{\beta}_{\mathfrak{s}}$ for $\beta < 5 - 2d$ $\mathbb{P}$-a.s. This is where the program in \cite{ZZ15} can stop and consider a solution 
\begin{equation*}
u  = z^{ 
\scalebox{0.18}{\begin{tikzpicture}
\draw[black, thick] (0.5,1) -- (0.5,0);
\filldraw[black] (0.5,1) circle (6pt); 
\end{tikzpicture}
}} + z^{ 
\scalebox{0.18}{\begin{tikzpicture}
\draw[black, thick] (-0.7,0.9) -- (0,0);
\draw[black, thick] (0.7,0.9) -- (0,0);
\draw[snake=zigzag] (0,0) -- (0,-1);
\filldraw[black] (-0.7,0.9) circle (6pt); 
\filldraw[black] (0.7,0.9) circle (6pt); 
\end{tikzpicture}
}} + z^{
\scalebox{0.16}{\begin{tikzpicture}
\draw[black, thick] (-0.7,0.9) -- (0,0);
\draw[black, thick] (0.7,0.9) -- (0,0);
\draw[black, thick] (0.7,0) -- (0,-0.9);
\draw[snake=zigzag](0,0) -- (0,-0.9);
\draw[snake=zigzag](0,-0.9) -- (0,-1.8);
\filldraw[black] (-0.7,0.9) circle (7pt); 
\filldraw[black] (0.7,0.9) circle (7pt); 
\filldraw[black] (0.7,0) circle (7pt); 
\end{tikzpicture}
}} + u_{3}. 
\end{equation*} 
\item To clarify where it goes wrong in the case $d =4$, we can take one step further and consider 
\begin{equation}\label{Define z4}
\partial_{t} z^{
\scalebox{0.14}{\begin{tikzpicture}
\filldraw[black] (-1.7,1.8) circle (7pt); 
\filldraw[black] (-0.5,1.8) circle (7pt); 
\filldraw[black] (0.5,1.8) circle (7pt); 
\filldraw[black] (1.7,1.8) circle (7pt); 
\draw[black, thick] (-1.1,0.9) -- (-1.7,1.8);
\draw[black, thick] (-1.1,0.9) -- (-0.5,1.8);
\draw[black, thick] (1.1,0.9) -- (1.7,1.8);
\draw[black, thick] (1.1,0.9) -- (0.5,1.8);
\draw[snake=zigzag](-1.1,0.9) -- (0,0);
\draw[snake=zigzag](1.1,0.9) -- (0,0);
\draw[snake=zigzag](0,0) -- (0,-0.8);
\end{tikzpicture}
}} - \Delta z^{
\scalebox{0.14}{\begin{tikzpicture}
\filldraw[black] (-1.7,1.8) circle (7pt); 
\filldraw[black] (-0.5,1.8) circle (7pt); 
\filldraw[black] (0.5,1.8) circle (7pt); 
\filldraw[black] (1.7,1.8) circle (7pt); 
\draw[black, thick] (-1.1,0.9) -- (-1.7,1.8);
\draw[black, thick] (-1.1,0.9) -- (-0.5,1.8);
\draw[black, thick] (1.1,0.9) -- (1.7,1.8);
\draw[black, thick] (1.1,0.9) -- (0.5,1.8);
\draw[snake=zigzag](-1.1,0.9) -- (0,0);
\draw[snake=zigzag](1.1,0.9) -- (0,0);
\draw[snake=zigzag](0,0) -- (0,-0.8);
\end{tikzpicture}
}} = - \mathbb{P}_{L} \divergence \left(  (z^{ 
\scalebox{0.18}{\begin{tikzpicture}
\draw[black, thick] (-0.7,0.9) -- (0,0);
\draw[black, thick] (0.7,0.9) -- (0,0);
\draw[snake=zigzag] (0,0) -- (0,-1);
\filldraw[black] (-0.7,0.9) circle (6pt); 
\filldraw[black] (0.7,0.9) circle (6pt); 
\end{tikzpicture}
}})^{\otimes 2} + 
z^{
\scalebox{0.16}{\begin{tikzpicture}
\draw[black, thick] (-0.7,0.9) -- (0,0);
\draw[black, thick] (0.7,0.9) -- (0,0);
\draw[black, thick] (0.7,0) -- (0,-0.9);
\draw[snake=zigzag](0,0) -- (0,-0.9);
\draw[snake=zigzag](0,-0.9) -- (0,-1.8);
\filldraw[black] (-0.7,0.9) circle (7pt); 
\filldraw[black] (0.7,0.9) circle (7pt); 
\filldraw[black] (0.7,0) circle (7pt); 
\end{tikzpicture}
}} \otimes z^{ 
\scalebox{0.18}{\begin{tikzpicture}
\draw[black, thick] (0.5,1) -- (0.5,0);
\filldraw[black] (0.5,1) circle (6pt); 
\end{tikzpicture}
}} + z^{ 
\scalebox{0.18}{\begin{tikzpicture}
\draw[black, thick] (0.5,1) -- (0.5,0);
\filldraw[black] (0.5,1) circle (6pt); 
\end{tikzpicture}
}} \otimes z^{
\scalebox{0.16}{\begin{tikzpicture}
\draw[black, thick] (-0.7,0.9) -- (0,0);
\draw[black, thick] (0.7,0.9) -- (0,0);
\draw[black, thick] (0.7,0) -- (0,-0.9);
\draw[snake=zigzag](0,0) -- (0,-0.9);
\draw[snake=zigzag](0,-0.9) -- (0,-1.8);
\filldraw[black] (-0.7,0.9) circle (7pt); 
\filldraw[black] (0.7,0.9) circle (7pt); 
\filldraw[black] (0.7,0) circle (7pt); 
\end{tikzpicture}
}} \right), \hspace{3mm} z^{
\scalebox{0.14}{\begin{tikzpicture}
\filldraw[black] (-1.7,1.8) circle (7pt); 
\filldraw[black] (-0.5,1.8) circle (7pt); 
\filldraw[black] (0.5,1.8) circle (7pt); 
\filldraw[black] (1.7,1.8) circle (7pt); 
\draw[black, thick] (-1.1,0.9) -- (-1.7,1.8);
\draw[black, thick] (-1.1,0.9) -- (-0.5,1.8);
\draw[black, thick] (1.1,0.9) -- (1.7,1.8);
\draw[black, thick] (1.1,0.9) -- (0.5,1.8);
\draw[snake=zigzag](-1.1,0.9) -- (0,0);
\draw[snake=zigzag](1.1,0.9) -- (0,0);
\draw[snake=zigzag](0,0) -- (0,-0.8);
\end{tikzpicture}
}} \rvert_{t=0} = 0, 
\end{equation} 
so that $u_{4} \triangleq u_{3} - z^{
\scalebox{0.14}{\begin{tikzpicture}
\filldraw[black] (-1.7,1.8) circle (7pt); 
\filldraw[black] (-0.5,1.8) circle (7pt); 
\filldraw[black] (0.5,1.8) circle (7pt); 
\filldraw[black] (1.7,1.8) circle (7pt); 
\draw[black, thick] (-1.1,0.9) -- (-1.7,1.8);
\draw[black, thick] (-1.1,0.9) -- (-0.5,1.8);
\draw[black, thick] (1.1,0.9) -- (1.7,1.8);
\draw[black, thick] (1.1,0.9) -- (0.5,1.8);
\draw[snake=zigzag](-1.1,0.9) -- (0,0);
\draw[snake=zigzag](1.1,0.9) -- (0,0);
\draw[snake=zigzag](0,0) -- (0,-0.8);
\end{tikzpicture}
}}$ satisfy 
\begin{align}\label{Define u4}
&\partial_{t} u_{4} - \Delta u_{4}  \nonumber \\
=& - \mathbb{P}_{L} \Bigg( \left( u_{4} + z^{
\scalebox{0.14}{\begin{tikzpicture}
\filldraw[black] (-1.7,1.8) circle (7pt); 
\filldraw[black] (-0.5,1.8) circle (7pt); 
\filldraw[black] (0.5,1.8) circle (7pt); 
\filldraw[black] (1.7,1.8) circle (7pt); 
\draw[black, thick] (-1.1,0.9) -- (-1.7,1.8);
\draw[black, thick] (-1.1,0.9) -- (-0.5,1.8);
\draw[black, thick] (1.1,0.9) -- (1.7,1.8);
\draw[black, thick] (1.1,0.9) -- (0.5,1.8);
\draw[snake=zigzag](-1.1,0.9) -- (0,0);
\draw[snake=zigzag](1.1,0.9) -- (0,0);
\draw[snake=zigzag](0,0) -- (0,-0.8);
\end{tikzpicture}
}} + z^{
\scalebox{0.16}{\begin{tikzpicture}
\draw[black, thick] (-0.7,0.9) -- (0,0);
\draw[black, thick] (0.7,0.9) -- (0,0);
\draw[black, thick] (0.7,0) -- (0,-0.9);
\draw[snake=zigzag](0,0) -- (0,-0.9);
\draw[snake=zigzag](0,-0.9) -- (0,-1.8);
\filldraw[black] (-0.7,0.9) circle (7pt); 
\filldraw[black] (0.7,0.9) circle (7pt); 
\filldraw[black] (0.7,0) circle (7pt); 
\end{tikzpicture}
}} \right)^{\otimes 2} + \left( u_{4} + z^{
\scalebox{0.14}{\begin{tikzpicture}
\filldraw[black] (-1.7,1.8) circle (7pt); 
\filldraw[black] (-0.5,1.8) circle (7pt); 
\filldraw[black] (0.5,1.8) circle (7pt); 
\filldraw[black] (1.7,1.8) circle (7pt); 
\draw[black, thick] (-1.1,0.9) -- (-1.7,1.8);
\draw[black, thick] (-1.1,0.9) -- (-0.5,1.8);
\draw[black, thick] (1.1,0.9) -- (1.7,1.8);
\draw[black, thick] (1.1,0.9) -- (0.5,1.8);
\draw[snake=zigzag](-1.1,0.9) -- (0,0);
\draw[snake=zigzag](1.1,0.9) -- (0,0);
\draw[snake=zigzag](0,0) -- (0,-0.8);
\end{tikzpicture}
}} + z^{
\scalebox{0.16}{\begin{tikzpicture}
\draw[black, thick] (-0.7,0.9) -- (0,0);
\draw[black, thick] (0.7,0.9) -- (0,0);
\draw[black, thick] (0.7,0) -- (0,-0.9);
\draw[snake=zigzag](0,0) -- (0,-0.9);
\draw[snake=zigzag](0,-0.9) -- (0,-1.8);
\filldraw[black] (-0.7,0.9) circle (7pt); 
\filldraw[black] (0.7,0.9) circle (7pt); 
\filldraw[black] (0.7,0) circle (7pt); 
\end{tikzpicture}
}} \right) \otimes z^{ 
\scalebox{0.18}{\begin{tikzpicture}
\draw[black, thick] (-0.7,0.9) -- (0,0);
\draw[black, thick] (0.7,0.9) -- (0,0);
\draw[snake=zigzag] (0,0) -- (0,-1);
\filldraw[black] (-0.7,0.9) circle (6pt); 
\filldraw[black] (0.7,0.9) circle (6pt); 
\end{tikzpicture}
}} + z^{ 
\scalebox{0.18}{\begin{tikzpicture}
\draw[black, thick] (-0.7,0.9) -- (0,0);
\draw[black, thick] (0.7,0.9) -- (0,0);
\draw[snake=zigzag] (0,0) -- (0,-1);
\filldraw[black] (-0.7,0.9) circle (6pt); 
\filldraw[black] (0.7,0.9) circle (6pt); 
\end{tikzpicture}
}} \otimes \left( u_{4} + z^{
\scalebox{0.14}{\begin{tikzpicture}
\filldraw[black] (-1.7,1.8) circle (7pt); 
\filldraw[black] (-0.5,1.8) circle (7pt); 
\filldraw[black] (0.5,1.8) circle (7pt); 
\filldraw[black] (1.7,1.8) circle (7pt); 
\draw[black, thick] (-1.1,0.9) -- (-1.7,1.8);
\draw[black, thick] (-1.1,0.9) -- (-0.5,1.8);
\draw[black, thick] (1.1,0.9) -- (1.7,1.8);
\draw[black, thick] (1.1,0.9) -- (0.5,1.8);
\draw[snake=zigzag](-1.1,0.9) -- (0,0);
\draw[snake=zigzag](1.1,0.9) -- (0,0);
\draw[snake=zigzag](0,0) -- (0,-0.8);
\end{tikzpicture}
}} + z^{
\scalebox{0.16}{\begin{tikzpicture}
\draw[black, thick] (-0.7,0.9) -- (0,0);
\draw[black, thick] (0.7,0.9) -- (0,0);
\draw[black, thick] (0.7,0) -- (0,-0.9);
\draw[snake=zigzag](0,0) -- (0,-0.9);
\draw[snake=zigzag](0,-0.9) -- (0,-1.8);
\filldraw[black] (-0.7,0.9) circle (7pt); 
\filldraw[black] (0.7,0.9) circle (7pt); 
\filldraw[black] (0.7,0) circle (7pt); 
\end{tikzpicture}
}} \right)   \nonumber \\
&+ \left( u_{4} + z^{
\scalebox{0.14}{\begin{tikzpicture}
\filldraw[black] (-1.7,1.8) circle (7pt); 
\filldraw[black] (-0.5,1.8) circle (7pt); 
\filldraw[black] (0.5,1.8) circle (7pt); 
\filldraw[black] (1.7,1.8) circle (7pt); 
\draw[black, thick] (-1.1,0.9) -- (-1.7,1.8);
\draw[black, thick] (-1.1,0.9) -- (-0.5,1.8);
\draw[black, thick] (1.1,0.9) -- (1.7,1.8);
\draw[black, thick] (1.1,0.9) -- (0.5,1.8);
\draw[snake=zigzag](-1.1,0.9) -- (0,0);
\draw[snake=zigzag](1.1,0.9) -- (0,0);
\draw[snake=zigzag](0,0) -- (0,-0.8);
\end{tikzpicture}
}}  \right) \otimes z^{ 
\scalebox{0.18}{\begin{tikzpicture}
\draw[black, thick] (0.5,1) -- (0.5,0);
\filldraw[black] (0.5,1) circle (6pt); 
\end{tikzpicture}
}}+ z^{ 
\scalebox{0.18}{\begin{tikzpicture}
\draw[black, thick] (0.5,1) -- (0.5,0);
\filldraw[black] (0.5,1) circle (6pt); 
\end{tikzpicture}
}} \otimes  \left( u_{4} + z^{
\scalebox{0.14}{\begin{tikzpicture}
\filldraw[black] (-1.7,1.8) circle (7pt); 
\filldraw[black] (-0.5,1.8) circle (7pt); 
\filldraw[black] (0.5,1.8) circle (7pt); 
\filldraw[black] (1.7,1.8) circle (7pt); 
\draw[black, thick] (-1.1,0.9) -- (-1.7,1.8);
\draw[black, thick] (-1.1,0.9) -- (-0.5,1.8);
\draw[black, thick] (1.1,0.9) -- (1.7,1.8);
\draw[black, thick] (1.1,0.9) -- (0.5,1.8);
\draw[snake=zigzag](-1.1,0.9) -- (0,0);
\draw[snake=zigzag](1.1,0.9) -- (0,0);
\draw[snake=zigzag](0,0) -- (0,-0.8);
\end{tikzpicture}
}}  \right) \Bigg).
\end{align}
The previous observation that $\divergence (z^{ 
\scalebox{0.18}{\begin{tikzpicture}
\draw[black, thick] (-0.7,0.9) -- (0,0);
\draw[black, thick] (0.7,0.9) -- (0,0);
\draw[snake=zigzag] (0,0) -- (0,-1);
\filldraw[black] (-0.7,0.9) circle (6pt); 
\filldraw[black] (0.7,0.9) circle (6pt); 
\end{tikzpicture}
}})^{\otimes 2}$, $\divergence ( z^{
\scalebox{0.16}{\begin{tikzpicture}
\draw[black, thick] (-0.7,0.9) -- (0,0);
\draw[black, thick] (0.7,0.9) -- (0,0);
\draw[black, thick] (0.7,0) -- (0,-0.9);
\draw[snake=zigzag](0,0) -- (0,-0.9);
\draw[snake=zigzag](0,-0.9) -- (0,-1.8);
\filldraw[black] (-0.7,0.9) circle (7pt); 
\filldraw[black] (0.7,0.9) circle (7pt); 
\filldraw[black] (0.7,0) circle (7pt); 
\end{tikzpicture}
}} \otimes z^{ 
\scalebox{0.18}{\begin{tikzpicture}
\draw[black, thick] (0.5,1) -- (0.5,0);
\filldraw[black] (0.5,1) circle (6pt); 
\end{tikzpicture}
}})$, and $\divergence( z^{ 
\scalebox{0.18}{\begin{tikzpicture}
\draw[black, thick] (0.5,1) -- (0.5,0);
\filldraw[black] (0.5,1) circle (6pt); 
\end{tikzpicture}
}} \otimes z^{
\scalebox{0.14}{\begin{tikzpicture}
\filldraw[black] (-1.7,1.8) circle (7pt); 
\filldraw[black] (-0.5,1.8) circle (7pt); 
\filldraw[black] (0.5,1.8) circle (7pt); 
\filldraw[black] (1.7,1.8) circle (7pt); 
\draw[black, thick] (-1.1,0.9) -- (-1.7,1.8);
\draw[black, thick] (-1.1,0.9) -- (-0.5,1.8);
\draw[black, thick] (1.1,0.9) -- (1.7,1.8);
\draw[black, thick] (1.1,0.9) -- (0.5,1.8);
\draw[snake=zigzag](-1.1,0.9) -- (0,0);
\draw[snake=zigzag](1.1,0.9) -- (0,0);
\draw[snake=zigzag](0,0) -- (0,-0.8);
\end{tikzpicture}
}})$ all have regularity of $\mathscr{C}^{\beta}_{\mathfrak{s}}$ for $\beta < 5 - 2d$, indicate that 
\begin{equation}\label{regularity of z4}
z^{
\scalebox{0.14}{\begin{tikzpicture}
\filldraw[black] (-1.7,1.8) circle (7pt); 
\filldraw[black] (-0.5,1.8) circle (7pt); 
\filldraw[black] (0.5,1.8) circle (7pt); 
\filldraw[black] (1.7,1.8) circle (7pt); 
\draw[black, thick] (-1.1,0.9) -- (-1.7,1.8);
\draw[black, thick] (-1.1,0.9) -- (-0.5,1.8);
\draw[black, thick] (1.1,0.9) -- (1.7,1.8);
\draw[black, thick] (1.1,0.9) -- (0.5,1.8);
\draw[snake=zigzag](-1.1,0.9) -- (0,0);
\draw[snake=zigzag](1.1,0.9) -- (0,0);
\draw[snake=zigzag](0,0) -- (0,-0.8);
\end{tikzpicture}
}} \in \mathscr{C}^{\beta}_{\mathfrak{s}}\text{ for }\beta < 7 - 2d \hspace{1mm}\mathbb{P}\text{-a.s.}
\end{equation} 
The most singular terms in \eqref{Define u4} are 
\begin{align*}
\divergence ( z^{ 
\scalebox{0.18}{\begin{tikzpicture}
\draw[black, thick] (-0.7,0.9) -- (0,0);
\draw[black, thick] (0.7,0.9) -- (0,0);
\draw[snake=zigzag] (0,0) -- (0,-1);
\filldraw[black] (-0.7,0.9) circle (6pt); 
\filldraw[black] (0.7,0.9) circle (6pt); 
\end{tikzpicture}
}} \otimes z^{
\scalebox{0.16}{\begin{tikzpicture}
\draw[black, thick] (-0.7,0.9) -- (0,0);
\draw[black, thick] (0.7,0.9) -- (0,0);
\draw[black, thick] (0.7,0) -- (0,-0.9);
\draw[snake=zigzag](0,0) -- (0,-0.9);
\draw[snake=zigzag](0,-0.9) -- (0,-1.8);
\filldraw[black] (-0.7,0.9) circle (7pt); 
\filldraw[black] (0.7,0.9) circle (7pt); 
\filldraw[black] (0.7,0) circle (7pt); 
\end{tikzpicture}
}}), \hspace{3mm} \divergence (z^{
\scalebox{0.16}{\begin{tikzpicture}
\draw[black, thick] (-0.7,0.9) -- (0,0);
\draw[black, thick] (0.7,0.9) -- (0,0);
\draw[black, thick] (0.7,0) -- (0,-0.9);
\draw[snake=zigzag](0,0) -- (0,-0.9);
\draw[snake=zigzag](0,-0.9) -- (0,-1.8);
\filldraw[black] (-0.7,0.9) circle (7pt); 
\filldraw[black] (0.7,0.9) circle (7pt); 
\filldraw[black] (0.7,0) circle (7pt); 
\end{tikzpicture}
}} \otimes z^{ 
\scalebox{0.18}{\begin{tikzpicture}
\draw[black, thick] (-0.7,0.9) -- (0,0);
\draw[black, thick] (0.7,0.9) -- (0,0);
\draw[snake=zigzag] (0,0) -- (0,-0.7);
\filldraw[black] (-0.7,0.9) circle (6pt); 
\filldraw[black] (0.7,0.9) circle (6pt); 
\end{tikzpicture}
}}), \hspace{3mm} \divergence (z^{
\scalebox{0.14}{\begin{tikzpicture}
\filldraw[black] (-1.7,1.8) circle (7pt); 
\filldraw[black] (-0.5,1.8) circle (7pt); 
\filldraw[black] (0.5,1.8) circle (7pt); 
\filldraw[black] (1.7,1.8) circle (7pt); 
\draw[black, thick] (-1.1,0.9) -- (-1.7,1.8);
\draw[black, thick] (-1.1,0.9) -- (-0.5,1.8);
\draw[black, thick] (1.1,0.9) -- (1.7,1.8);
\draw[black, thick] (1.1,0.9) -- (0.5,1.8);
\draw[snake=zigzag](-1.1,0.9) -- (0,0);
\draw[snake=zigzag](1.1,0.9) -- (0,0);
\draw[snake=zigzag](0,0) -- (0,-0.8);
\end{tikzpicture}
}}\otimes z^{ 
\scalebox{0.18}{\begin{tikzpicture}
\draw[black, thick] (0.5,1) -- (0.5,0);
\filldraw[black] (0.5,1) circle (6pt); 
\end{tikzpicture}
}}), \hspace{1mm} \text{ and } \hspace{1mm} \divergence (z^{ 
\scalebox{0.18}{\begin{tikzpicture}
\draw[black, thick] (0.5,1) -- (0.5,0);
\filldraw[black] (0.5,1) circle (6pt); 
\end{tikzpicture}
}} \otimes z^{
\scalebox{0.14}{\begin{tikzpicture}
\filldraw[black] (-1.7,1.8) circle (7pt); 
\filldraw[black] (-0.5,1.8) circle (7pt); 
\filldraw[black] (0.5,1.8) circle (7pt); 
\filldraw[black] (1.7,1.8) circle (7pt); 
\draw[black, thick] (-1.1,0.9) -- (-1.7,1.8);
\draw[black, thick] (-1.1,0.9) -- (-0.5,1.8);
\draw[black, thick] (1.1,0.9) -- (1.7,1.8);
\draw[black, thick] (1.1,0.9) -- (0.5,1.8);
\draw[snake=zigzag](-1.1,0.9) -- (0,0);
\draw[snake=zigzag](1.1,0.9) -- (0,0);
\draw[snake=zigzag](0,0) -- (0,-0.8);
\end{tikzpicture}
}}), 
\end{align*}
all of which have homogeneity of $7 - \frac{5d}{2}$, implying that $u_{4} \in \mathscr{C}^{\beta}_{\mathfrak{s}}$ for $\beta < 9 - \frac{5d}{2}$, once again failing to be a function for $d =4$ (in fact, for all $d\geq 4$). 
\end{enumerate} 

Denoting for simplicity 
\begin{align*}
z_{1} \triangleq z^{ 
\scalebox{0.18}{\begin{tikzpicture}
\draw[black, thick] (0.5,1) -- (0.5,0);
\filldraw[black] (0.5,1) circle (6pt); 
\end{tikzpicture}
}}, \hspace{2mm} z_{2} \triangleq z^{ 
\scalebox{0.18}{\begin{tikzpicture}
\draw[black, thick] (-0.7,0.9) -- (0,0);
\draw[black, thick] (0.7,0.9) -- (0,0);
\draw[snake=zigzag] (0,0) -- (0,-1);
\filldraw[black] (-0.7,0.9) circle (6pt); 
\filldraw[black] (0.7,0.9) circle (6pt); 
\end{tikzpicture}
}},\hspace{2mm}  z_{3} \triangleq z^{
\scalebox{0.16}{\begin{tikzpicture}
\draw[black, thick] (-0.7,0.9) -- (0,0);
\draw[black, thick] (0.7,0.9) -- (0,0);
\draw[black, thick] (0.7,0) -- (0,-0.9);
\draw[snake=zigzag](0,0) -- (0,-0.9);
\draw[snake=zigzag](0,-0.9) -- (0,-1.8);
\filldraw[black] (-0.7,0.9) circle (7pt); 
\filldraw[black] (0.7,0.9) circle (7pt); 
\filldraw[black] (0.7,0) circle (7pt); 
\end{tikzpicture}
}},  \hspace{1mm} \text{ and } \hspace{1mm} z_{4} \triangleq z^{
\scalebox{0.14}{\begin{tikzpicture}
\filldraw[black] (-1.7,1.8) circle (7pt); 
\filldraw[black] (-0.5,1.8) circle (7pt); 
\filldraw[black] (0.5,1.8) circle (7pt); 
\filldraw[black] (1.7,1.8) circle (7pt); 
\draw[black, thick] (-1.1,0.9) -- (-1.7,1.8);
\draw[black, thick] (-1.1,0.9) -- (-0.5,1.8);
\draw[black, thick] (1.1,0.9) -- (1.7,1.8);
\draw[black, thick] (1.1,0.9) -- (0.5,1.8);
\draw[snake=zigzag](-1.1,0.9) -- (0,0);
\draw[snake=zigzag](1.1,0.9) -- (0,0);
\draw[snake=zigzag](0,0) -- (0,-0.8);
\end{tikzpicture}
}}, 
\end{align*}
let us collect our observations concerning the regularity of $z_{k}$'s because the process stops when all the products are \emph{almost} well-defined. Informally this requires $z_{k}$ to have the regularity of $\mathscr{C}^{\beta}_{\mathfrak{s}}$ for $\beta \geq 0$ because otherwise, we will need to continue to subtract out $z_{k} \otimes z_{k}$. 
\begin{table}[H]
\begin{tabular}{ |p{0.3cm}||p{8cm}|p{3cm}|p{1cm}}
 \hline
$k$ & The most singular term in $z_{k}$-equation, its regularity $\mathscr{C}^{\beta}_{\mathfrak{s}}$ & Regularity $\mathscr{C}^{\beta}_{\mathfrak{s}}$ of $z_{k}$  \\
 \hline
1 &  $\xi, \hspace{59mm} \beta < -\frac{d+2}{2}$& $ \beta < 1-\frac{d}{2}$  \\
\hline
2 &  $\divergence (z_{1})^{\otimes 2}, \hspace{48mm} \beta < 1-d$& $\beta < 3-d$  \\
\hline
3& $\divergence (z_{2} \otimes z_{1} + z_{1} \otimes z_{2}), \hspace{31mm} \beta < 3 - \frac{3d}{2}$ & $\beta < 5- \frac{3d}{2}$  \\
\hline
4 & $\divergence (z_{2}^{\otimes 2} + z_{3} \otimes z_{1} + z_{1} \otimes z_{3}), \hspace{23mm} \beta < 5 - 2d$ & $\beta < 7 - 2d$\\
\hline 
\end{tabular}
\end{table}
\noindent At every step, the regularity of $z_{k}$ seems to get better by $2 - \frac{d}{2}$, which in fact reveals that in case $d = 4$, the regularity of $z_{k}$ is not getting better. This is why, informally, we say that ``\emph{the trees become infinite}'' as the locally subcritical case approaches the critical case. In locally critical or supercritical cases, such iterations do not lead to any gain in regularity. 

Another important follow-up concerning this typical method of constructing a solution theory is the need to subtract out renormalization constants. One can directly solve \eqref{Define z1} so that 
\begin{equation}\label{Explicit solution}
z^{ 
\scalebox{0.18}{\begin{tikzpicture}
\draw[black, thick] (0.5,1) -- (0.5,0);
\filldraw[black] (0.5,1) circle (6pt); 
\end{tikzpicture}
}}(t) = \int_{0}^{t} e^{(t-s) \Delta} \xi(s) ds. 
\end{equation}
Informally stated, this implies that $z^{ 
\scalebox{0.18}{\begin{tikzpicture}
\draw[black, thick] (-0.7,0.9) -- (0,0);
\draw[black, thick] (0.7,0.9) -- (0,0);
\draw[snake=zigzag] (0,0) -- (0,-0.7);
\filldraw[black] (-0.7,0.9) circle (6pt); 
\filldraw[black] (0.7,0.9) circle (6pt); 
\end{tikzpicture}
}}$, according to \eqref{Define z2}, consists of a product of the STWN $\xi$, $z^{
\scalebox{0.16}{\begin{tikzpicture}
\draw[black, thick] (-0.7,0.9) -- (0,0);
\draw[black, thick] (0.7,0.9) -- (0,0);
\draw[black, thick] (0.7,0) -- (0,-0.9);
\draw[snake=zigzag](0,0) -- (0,-0.9);
\draw[snake=zigzag](0,-0.9) -- (0,-1.8);
\filldraw[black] (-0.7,0.9) circle (7pt); 
\filldraw[black] (0.7,0.9) circle (7pt); 
\filldraw[black] (0.7,0) circle (7pt); 
\end{tikzpicture}
}}$ has a cubic of the STWN $\xi$, and consequently the equation of $u_{3}$ has plenty of various products of the STWN. Such products must be computed using Wick products (e.g. \cite{J97}), and they lead to the appropriate choices of renormalization constants (see \cite{Y21c} for details), typically as the zeroth order Wiener chaos (e.g. \cite[Equations (180), (184), and (192)]{Y21a} or \cite[Equation (5.7)]{Y23a}). 
\end{example} 

Very recently, the study of stochastic Yang-Mills equation has caught much attention; in coordinates, this equation reads 
\begin{equation}\label{Yang-Mills}
\partial_{t} A_{i} = \Delta A_{i} + \xi_{i} + \left[ A_{j}, 2 \partial_{j} A_{i} - \partial_{i} A_{j} + [A_{j}, A_{i} ] \right], 
\end{equation} 
where $\xi$ is a STWN, $[A, B] = AB- BA$ is a Lie bracket, and the solution is Lie-algebra-valued. For a precise description with more details, we refer to \cite[Equation (1.7)]{CCHS22a}. Following the approach of stochastic quantization that was pioneered by Parisi and Wu \cite{PW81}, there is a growing interest in constructing an invariant measure for \eqref{Yang-Mills}. Via the theory of regularity structures, Chandra, Chevyrev, Hairer, and Shen made significant progress on the analysis of \eqref{Yang-Mills} in both 2D and 3D cases, respectively in \cite{CCHS22a, CCHS22b}. We observe that the nonlinear terms of \eqref{Yang-Mills} consist informally of cubic and Burgers' types, specifically $[A_{j}, [A_{j}, A_{i} ]]$ and $[A_{j}, 2 \partial_{j} A_{i} - \partial_{i} A_{j} ]$, respectively. This makes it clear that our discussion concerning the local subcriticality in Example \ref{Example 2.1} essentially applies, making the case $d= 4$ locally critical and considerably more difficult. 

\subsection{Review on the technique of convex integration: deterministic case} 
The purpose of this section is to review recent developments on the technique of convex integration. To make our statement precise, we distinguish between two types of solutions to the 3D deterministic Navier-Stokes equations on a torus; it can be generalized in the presence of random noise or in case $x \in \mathbb{R}^{d}$ in a standard way.

\begin{define}\label{Definition 2.1}
\rm{(\hspace{1sp}\cite[Definitions 3.5 and 3.6]{BV19b})} 
\begin{enumerate}[label=(\roman*)]
\item A vector field $u \in C_{\text{weak}} ([0,\infty); L^{2} (\mathbb{T}^{3})) \cap L^{2} ( 0,\infty; \dot{H}^{1}(\mathbb{T}^{3}))$ is a Leray-Hopf weak solution to the deterministic Navier-Stokes equations if for any $t$, $u(\cdot, t)$ is weakly divergence-free, has zero mean, satisfies the Navier-Stokes equations distributionally, and satisfies the energy inequality 
\begin{equation*}
\frac{1}{2} \lVert u(t) \rVert_{L^{2}}^{2} + \int_{0}^{t} \lVert \nabla u (s) \rVert_{L^{2}}^{2} ds \leq \frac{1}{2} \lVert u(0) \rVert_{L^{2}}^{2}. 
\end{equation*} 
\item A vector field $u \in C([0, \infty); L^{2}(\mathbb{T}^{3}))$ is a weak solution to the deterministic Navier-Stokes equations if for any $t$, $u(\cdot, t)$ is weakly divergence-free, has zero mean, and satisfies the Navier-Stokes equations distributionally. 
\end{enumerate} 
\end{define} 
 
The origin of the convex integration technique is rooted in the work of Nash \cite{N54} concerning the $C^{1}$ isometric embedding theorem in geometry. This isometric embedding theorem of Nash was considered by Gromov to be one of the primary examples of homotopy-principle, who coined the name ``convex integration'' in \cite[Part 2.4]{G86}. The major breakthrough of its application to the deterministic Euler equations, \eqref{Navier-Stokes} without diffusion $\Delta u$ and zero noise, is credited to De Lellis and Sz$\acute{\mathrm{e}}$kelyhidi Jr. \cite{DS09} who used differential inclusion approach and Baire category argument to construct a solution in $L^{\infty} (0, \infty; L^{\infty} (\mathbb{R}^{d}))$ for all $d \geq 2$, with compact support in space-time, demonstrating non-uniqueness in comparison to a zero solution. Further refinements of the convex integration technique led Isett \cite{I18} to solve the Onsager's conjecture concerning the $C^{\frac{1}{3}}$-threshold of the deterministic Euler equations that conserves energy. Additionally, Buckmaster and Vicol \cite{BV19a} proved the non-uniqueness of weak solutions to the 3D deterministic Navier-Stokes equations by constructing a weak solution with prescribed energy; such results were extended up to the $L^{2}(\mathbb{T}^{3})$-supercritical threshold of $(-\Delta)^{m}$ for $m < \frac{5}{4}$ by Buckmaster, Colombo, and Vicol \cite{BCV22} and Luo and Titi \cite{LT20}. 
 
We briefly explain the key ideas of the convex integration technique in the case of \cite{BV19a}. We fix an arbitrary function $e: \mathbb{R}_{\geq 0} \mapsto \mathbb{R}_{ \geq 0}$ with minimum restrictions. Then, for all $q \in \mathbb{N}_{0}$, we aim to construct solutions to the following Navier-Stokes-Reynolds system:
 \begin{equation}\label{Navier-Stokes Reynolds}
\partial_{t} u_{q} + \divergence (u_{q} \otimes u_{q}) + \nabla \pi_{q} = \Delta u_{q} + \divergence \mathring{R}_{q}, \hspace{3mm} \nabla\cdot u_{q} = 0, 
\end{equation} 
that satisfies several inductive hypothesis, two of them being 
\begin{equation}\label{inductive hypothesis}
\lVert \mathring{R}_{q} \rVert_{C_{t} L_{x}^{1}} \leq \delta_{q+1}, \hspace{3mm} \left\lvert e(t) - \lVert u_{q}(t) \rVert_{L^{2}}^{2} \right\rvert \leq \delta_{q+1} 
\end{equation} 
where $\{\delta_{q}\}_{q\in\mathbb{N}_{0}}$ is a strictly decreasing sequence in $\mathbb{R}_{+}$ such that $\lim_{q\to\infty} \delta_{q} = 0$. Such hypothesis assures that the limiting solution solves the deterministic Navier-Stokes equations with the prescribed energy. One verifies the existence of $(u_{0}, \pi_{0}, \mathring{R}_{0})$ that satisfies \eqref{Navier-Stokes Reynolds} and the inductive hypothesis, assumes the existence of such $(u_{q}, \pi_{q}, \mathring{R}_{q})$, and aims to construct $(u_{q+1}, \pi_{q+1}, \mathring{R}_{q+1})$ accordingly. In order to do so, based on the nonlinear term $\divergence (u_{q} \otimes u_{q})$, one would construct the perturbation 
\begin{equation}\label{Define perturbation}
w_{q+1} \triangleq u_{q+1} - u_{q}, 
\end{equation} 
and solve for 
\begin{align}\label{Define Rq+1}
\mathring{R}_{q+1} =& \text{div}^{-1} \Bigg( \divergence \mathring{R}_{q} + \partial_{t} w_{q+1} \nonumber \\
&\hspace{4mm} + \divergence \left( u_{q} \otimes w_{q+1} + w_{q+1} \otimes u_{q} + w_{q+1} \otimes w_{q+1}  \right) + \nabla (\pi_{q+1} - \pi_{q}) - \Delta w_{q+1}  \Bigg). 
\end{align} 
The heart of the matter is to construct $w_{q+1}$ such that $w_{q+1} \otimes w_{q+1}$ cancels out the previous error $\mathring{R}_{q}$ in \eqref{Define Rq+1} and allow the new error $\mathring{R}_{q+1}$ to be smaller than $\mathring{R}_{q}$, as required in \eqref{inductive hypothesis}. 

\begin{remark}
As we pointed out in Remark \ref{Remark 2.2}, all classical approaches to construct solutions start from approximating the given PDEs, e.g. Galerkin approximation, mollifier, or Fourier truncation. So does the convex integration technique, as can be seen in \eqref{Navier-Stokes Reynolds}. The crucial difference is that in the classical approaches, all the solutions to the approximations of the original PDEs satisfy the energy inequality at every step and hence so do their limits. The approximate solution $u_{q}$ to \eqref{Navier-Stokes Reynolds} does not satisfy any energy inequality; in fact, its energy is \emph{almost} the prescribed arbitrary energy profile $e(t)$, modulus a small vanishing error according to \eqref{inductive hypothesis}. 
\end{remark} 
We refer to \cite{BV19b} for a survey of further developments on the technique of convex integration in the deterministic setting. 

\subsection{Review on the technique of convex integration: stochastic case}\label{Section 2.3}
In the stochastic setting, the technique of convex integration was initially applied on certain compressible Euler equations in \cite{BFH20} by Breit, Feireisl, and Hofmanov$\acute{\mathrm{a}}$ and \cite{CFF19} by Chiodaroli, Feireisl, and Flandoli in two and three dimensions to prove non-uniqueness path-wise. At this point, the authors of \cite{BFH20, CFF19} already observed that such a solution constructed via convex integration is probabilistically strong, which informally requires that the solution is adapted to the filtration generated by the \emph{a priori} fixed noise. We emphasize that the random noise considered in \cite{BFH20, CFF19} were white-in-time and all other works described in this Section \ref{Section 2.3} had random noise that was white only in time. 

Subsequently, Hofmanov$\acute{\mathrm{a}}$, Zhu, and Zhu \cite{HZZ19} applied the convex integration scheme that is similar to \cite[Section 7]{BV19b} to prove non-uniqueness in law of the 3D Navier-Stokes equations forced by additive, linear multiplicative, or nonlinear noise (see also \cite{Y22a} on its generalization akin to \cite{BCV22}). The importance of \cite{HZZ19} is three-fold, at least. First, it was the first construction of a probabilistically strong solution to the 3D Navier-Stokes equations forced by random noise, which was an open problem for a long time (e.g. \cite{F08}). Second, the classical Yamada-Watanabe theorem states that path-wise uniqueness implies uniqueness in law, indicating that the non-uniqueness obtained in \cite{HZZ19} is stronger than those in \cite{BFH20, CFF19}, and the addition of diffusion in the Navier-Stokes equations certainly makes the proof of non-uniqueness more difficult. Third, Cherney's theorem \cite[ Theorem 3.2]{C03} implies that the existence of a probabilistically strong solution and uniqueness in law together imply path-wise uniqueness. Consequently, the construction of a probabilistically strong solution that is non-unique in law by \cite{HZZ19} completely wiped out the possibility of proving path-wise uniqueness of the solution to the 3D Navier-Stokes equations forced by random noise via an application of Cherney's theorem. 

\begin{remark}
We highlight one particular research direction on extension of \cite{HZZ19}. Even in the deterministic setting, the uniqueness of a Leray-Hopf weak solution to the 3D Navier-Stokes equations remains an outstanding open problem (recall Definition \ref{Definition 2.1}). Inspired by \cite{ABC22} due to  Albritton, Bru$\acute{e}$, and Colombo, which demonstrated the non-uniqueness of the Leray-Hopf weak solution to the 3D deterministic Navier-Stokes equations with non-zero force, \cite{BJLZ24, HZZ23a} proved its extension in the presence of random noise. Intuitively, a favorable choice of force may help one prove non-uniqueness; however, it is not clear if the random noise in a classical form such as the typical cylindrical Wiener process in \cite[Theorem 1.1]{HZZ19} or STWN $\xi$ in \eqref{Navier-Stokes} can help induce non-uniqueness of the solution. 

On a related note, it has been well documented that certain types of random noise possess regularizing effects. For example, Kim \cite{K11} showed that certain multiplicative noise can have damping effects inducing global well-posedness of non-diffusive PDEs starting from all sufficiently small initial data with probability arbitrarily close to one. Moreover, transport noise allowed Flandoli, Gubinelli, and Priola \cite{FGP10} to break the threshold of regularity criterion for continuity equation, set previously by Di Perna and Lions \cite{DL89} in the deterministic setting. At the time of writing this manuscript, the convex integration technique seems to induce non-uniqueness even in the presence of such random noise with regularizing property (see e.g. \cite{HLP22, KY25} on the convex integration technique applied on SPDEs with transport noise). 
\end{remark} 

There have been many more extensions and improvements concerning such area of convex integration on SPDEs with white-in-time noise; we refer to \cite{Y25e} for further discussions and details. 

\section{Convex integration technique applied to singular SPDEs}
The purpose of this section is to describe two remarkable works of convex integration technique applied to singular SPDEs: \cite{HZZ21b} on the 3D Navier-Stokes equations (\hspace{1sp}\cite{LZ23, LZ25} on the 2D case); \cite{HZZ22a, HLZZ23} on the 2D surface quasi-geostrophic (SQG) equation. 

\begin{example}\label{Example on singular NS}
Hofmanov$\acute{\mathrm{a}}$, Zhu, and Zhu \cite[Theorem 1.1 and Corollary 1.2]{HZZ21b} were the first to combine the technique of convex integration with the theory of paracontrolled distributions to construct infinitely many analytically weak solutions from given initial data and consequently prove non-uniqueness in law of the 3D Navier-Stokes equations forced by STWN, precisely \eqref{Navier-Stokes}. The main idea can be readily explained from our computations in Example \ref{Example 2.1}. In fact, Hofmanov$\acute{\mathrm{a}}$, Zhu, and Zhu stopped at step (ii) in Example \ref{Example 2.1} because $u_{2}$ is already a function and defined a solution as 
\begin{equation}\label{solution}
u  = z^{ 
\scalebox{0.18}{\begin{tikzpicture}
\draw[black, thick] (0.5,1) -- (0.5,0);
\filldraw[black] (0.5,1) circle (6pt); 
\end{tikzpicture}
}} + z^{ 
\scalebox{0.18}{\begin{tikzpicture}
\draw[black, thick] (-0.7,0.9) -- (0,0);
\draw[black, thick] (0.7,0.9) -- (0,0);
\draw[snake=zigzag] (0,0) -- (0,-1);
\filldraw[black] (-0.7,0.9) circle (6pt); 
\filldraw[black] (0.7,0.9) circle (6pt); 
\end{tikzpicture}
}} + u_{2}.
\end{equation} 
They split $u_{2} = v^{1} + v^{2}$ appropriately so that 
\begin{align*}
\partial_{t} v^{1} + \mathbb{P}_{L} \divergence V^{1} = \Delta v^{1},  \hspace{5mm} \partial_{t} v^{2} + \mathbb{P}_{L} \divergence V^{2} = \Delta v^{2}, 
\end{align*}
where the general rule is 
\begin{enumerate}[label=(\alph*)]
\item to take irregular terms and ill-defined products to $V^{1}$ and regular and well-defined products to $V^{2}$, 
\item to make $V^{1}$ linear in $v^{1}$, and consequently $V^{2}$ nonlinear in $v^{1}$.  
\end{enumerate} 
There is an additional step of considering an appropriate paracontrolled ansatz; we skip this for brevity and refer the interested reader to \cite{HZZ21b} for details. Similarly to \eqref{Navier-Stokes Reynolds}-\eqref{Define Rq+1}, for all $q \in \mathbb{N}_{0}$, they considered 
\begin{subequations}
\begin{align}
& \partial_{t} v_{q}^{1} + \mathbb{P}_{L} \divergence V_{q}^{1} = \Delta v_{q}^{1},  \label{vq1}\\
& \partial_{t} v_{q}^{2} + \mathbb{P}_{L} \divergence V_{q}^{2} = \Delta v_{q}^{2} + \divergence \mathring{R}_{q}, \label{vq2}
\end{align}
\end{subequations} 
along with appropriate inductive hypothesis (cf. \eqref{inductive hypothesis}). Then they 
\begin{enumerate}
\item started the iteration from $v_{0}^{2} \equiv 0$, solved for $\mathring{R}_{0}$ similarly to \eqref{Define Rq+1}, obtained the unique $v_{0}^{1}$ by a  fixed point theorem taking advantage of the linearity of \eqref{vq1}, 
\item assumed the existence of unique $(v_{q}^{1}, v_{q}^{2}, \mathring{R}_{q})$ that satisfies the inductive hypothesis, 
\item constructed $w_{q+1}^{2} \triangleq v_{q+1}^{2} - v_{q}^{2}$ (cf. \eqref{Define perturbation}), 
\item plugged $v_{q+1}^{2} = v_{q}^{2} + w_{q+1}^{2}$ to \eqref{vq1}, and deduced the unique $v_{q+1}^{1}$ by a fixed point theorem taking advantage of linearity of \eqref{vq1} again.  
\end{enumerate} 
Before proceeding to step (5), we pause here to make a comment. 
\begin{remark}
Going back to \cite[p. 5]{GIP15}, the theory of paracontrolled distributions constructs an analytically weak solution. While $(v_{q}^{1}, v_{q}^{2}, \mathring{R}_{q})$ is smooth at every step $q\in \mathbb{N}_{0}$, because \eqref{vq1} consists of rough terms, one may question if the limit $v^{1}$ will have sufficient regularity to satisfy a weak formulation. In fact, one can recall from Example \ref{Example 2.1} that the most singular term in \eqref{Define u2} is 
$\divergence (z^{ 
\scalebox{0.18}{\begin{tikzpicture}
\draw[black, thick] (-0.7,0.9) -- (0,0);
\draw[black, thick] (0.7,0.9) -- (0,0);
\draw[snake=zigzag] (0,0) -- (0,-1);
\filldraw[black] (-0.7,0.9) circle (6pt); 
\filldraw[black] (0.7,0.9) circle (6pt); 
\end{tikzpicture}
}} \otimes z^{ 
\scalebox{0.18}{\begin{tikzpicture}
\draw[black, thick] (0.5,1) -- (0.5,0);
\filldraw[black] (0.5,1) circle (6pt); 
\end{tikzpicture}
}})$ of regularity $\mathscr{C}^{\beta}_{\mathfrak{s}}$ for $\beta < 3 - \frac{3d}{2} = - \frac{3}{2}$ so that we can expect at best $v^{1} \in \mathscr{C}^{\beta}_{\mathfrak{s}}$ for $\beta < \frac{1}{2}$ $\mathbb{P}$-a.s. (cf. \cite[Remark 5.3]{HZZ21b}), which is more than enough for an analytically weak solution because the Navier-Stokes type nonlinearity requires only $L^{2}(0, T; L^{2}(\mathbb{T}^{3}))$ for the purpose of distributional formulation. We also note that although the solution in \eqref{solution} is constructed to be analytically weak, $z^{ 
\scalebox{0.18}{\begin{tikzpicture}
\draw[black, thick] (0.5,1) -- (0.5,0);
\filldraw[black] (0.5,1) circle (6pt); 
\end{tikzpicture}
}}$
and $z^{ 
\scalebox{0.18}{\begin{tikzpicture}
\draw[black, thick] (-0.7,0.9) -- (0,0);
\draw[black, thick] (0.7,0.9) -- (0,0);
\draw[snake=zigzag] (0,0) -- (0,-1);
\filldraw[black] (-0.7,0.9) circle (6pt); 
\filldraw[black] (0.7,0.9) circle (6pt); 
\end{tikzpicture}
}}$ are solved in a mild formulation: \eqref{Explicit solution} and 
\begin{align*}
z^{ 
\scalebox{0.18}{\begin{tikzpicture}
\draw[black, thick] (-0.7,0.9) -- (0,0);
\draw[black, thick] (0.7,0.9) -- (0,0);
\draw[snake=zigzag] (0,0) -- (0,-1);
\filldraw[black] (-0.7,0.9) circle (6pt); 
\filldraw[black] (0.7,0.9) circle (6pt); 
\end{tikzpicture}
}} (t) = - \int_{0}^{t} e^{(t-s) \Delta} \mathbb{P}_{L} \divergence ( z^{ 
\scalebox{0.18}{\begin{tikzpicture}
\draw[black, thick] (0.5,1) -- (0.5,0);
\filldraw[black] (0.5,1) circle (6pt); 
\end{tikzpicture}
} \otimes 2} ) (s) ds. 
\end{align*}
\end{remark} 
We now resume the iteration procedure. After constructing the unique $v_{q+1}^{1}$ in step (4), Hofmanov$\acute{\mathrm{a}}$, Zhu, and Zhu
\begin{enumerate}[resume]
\item considered \eqref{vq2}, solved for $\mathring{R}_{q+1}$ similarly to \eqref{Define Rq+1},
\item and verified that the inductive hypothesis are satisfied to conclude the necessary steps at level $q+1$. 
\end{enumerate} 
Besides the elegant mathematics combining the theory of paracontrolled distributions and the technique of convex integration, the result of \cite{HZZ21b} is of significance because the solution constructed therein is global-in-time, while, as we emphasized, the theory of regularity structures or the theory of paracontrolled distributions alone led only to local solution theory in \cite{ZZ15}. While the 2D Navier-Stokes equations forced by STWN admits global-in-time solution theory due to the explicit knowledge of invariant measure \cite{DD02} or because the dimension is sufficiently low to make the diffusion energy critical \cite{HR24}, it was not known if one can extend \cite{ZZ15} to construct a global-in-time solution in the 3D case. \cite{HZZ21b} achieved this by the iteration procedure (1)-(6) and the argument from \cite{HZZ21a} that allows the solution to exist up to a  stopping time and repeat, which is possible because the stopping time depends only on the fixed noise, not the terminal data of the solution.
\end{example} 

\begin{remark}\label{Remark 3.3}
One of the key takeaways from Example \ref{Example on singular NS} can be applied to the solution theory of PDEs in general, not just stochastic. For example, Wu and the second author in \cite{WY24} studied the following deterministic hyperbolic Navier-Stokes equations: 
\begin{equation}\label{hyperbolic} 
\partial_{tt} u +\partial_{t} u + \divergence (u\otimes u) + \nabla \pi + (-\Delta)^{\frac{\gamma}{2}} u =0, \hspace{3mm} \nabla\cdot u = 0.
\end{equation} 
The origin of this PDE goes back three quarters of a century to \cite{C49, C58} by Cattaneo and \cite{V58} by Vernotte. For example, the second temporal derivative $\partial_{tt}u$ can be derived by considering delayed stress. Importantly for our discussion, although a local-in-time solution for \eqref{hyperbolic} can be constructed (e.g. \cite{JLTW23}), the global existence of even a weak solution for \eqref{hyperbolic} in any dimension, even two, was open because the solution $u$ does not satisfy the energy inequality in sharp contrast to the Navier-Stokes equations due to the problematic terms of $\partial_{tt}u \cdot u$ and $(u\cdot\nabla) u \cdot \partial_{t} u$ upon $L^{2}_{x}$-estimate of $u$ and $\partial_{t} u$, respectively. The absence of any bounded quantities makes the classical approach such as Galerkin approximation inapplicable. Therefore, this problem of constructing a global-in-time weak solution to \eqref{hyperbolic}  shares similarities with singular SPDEs (recall Example \ref{Example 2.1}). 

Fortunately, notwithstanding the sharp differences in the linear term $\partial_{tt} u$, the nonlinear term of the Navier-Stokes equations and \eqref{hyperbolic} are identical, suggesting a potential application of the convex integration technique to construct a global-in-time solution. In \cite{WY24}, Wu and the second author succeeded in doing so in the 2D case as long as $\gamma \in (0, \frac{4}{3})$; because the energy was prescribed, the solution therein was also non-unique (we refer to \cite[Theorem 2.1]{WY24} for details). This is an instance in which the technique of convex integration was utilized not with the goal of proving non-uniqueness of a weak solution that is already known to exist (e.g. \cite{BV19a}); instead, it was applied to a PDE for which no known method (Galerkin approximation, mollifier, or Fourier truncation) can construct a global-in-time weak solution, similarly to Example \ref{Example on singular NS}. 
\end{remark} 

\begin{example}\label{Example on singular SQG} 
Hofmanov$\acute{\mathrm{a}}$, Zhu, and Zhu \cite{HZZ22a} were the first to apply the convex integration technique to construct solution to locally critical/supercritical singular SPDEs, namely the following singular SQG equation: for $\alpha \in [0, 1), \gamma \in [0, \frac{3}{2})$, and $x \in \mathbb{T}^{2}$, 
\begin{equation}\label{SQG}
\partial_{t} \theta + \mathcal{R}^{\bot} \theta \cdot \nabla  \theta + (-\Delta)^{\frac{\gamma}{2}} \theta = (-\Delta)^{\frac{\alpha}{2}} \xi, \hspace{3mm} \mathcal{R}^{\bot} \theta \triangleq \nabla^{\bot} ( -\Delta)^{-\frac{1}{2}} \theta,
\end{equation} 
where $\mathcal{R}$ represents the Riesz transform vector, and $\xi$ was white-in-space but not allowed to be white in time due to a technical restriction. The main result \cite[Theorems 1.2 and 3.2]{HZZ22a} is the construction of infinitely many analytically weak solutions from a given initial data $\theta^{\text{in}}$, which again leads to non-uniqueness in law; specifically, the solution $\theta$ satisfies 
\begin{align}\label{analytically weak}
&\langle \theta(t), \psi \rangle_{\dot{H}^{-\frac{1}{2}} \times \dot{H}^{\frac{1}{2}}} + \int_{0}^{t} \frac{1}{2} \langle \Lambda^{-\frac{1}{2}} \theta, \Lambda^{\frac{1}{2}} [ \mathcal{R}^{\bot} \cdot \nabla, \nabla \psi] \theta \rangle_{\dot{H}^{-\frac{1}{2}} \times \dot{H}^{\frac{1}{2}}} ds  \\
&+ \int_{0}^{t} \langle (-\Delta)^{\frac{\gamma}{2}} \theta, \psi \rangle_{\dot{H}^{-\frac{1}{2}} \times \dot{H}^{\frac{1}{2}}} ds = \langle \theta^{\text{in}}, \psi \rangle_{\dot{H}^{-\frac{1}{2}} \times \dot{H}^{\frac{1}{2}}} + \int_{0}^{t} \langle (-\Delta)^{\frac{\alpha}{2}} \xi, \psi \rangle_{\dot{H}^{-\frac{1}{2}} \times \dot{H}^{\frac{1}{2}}} ds \nonumber
\end{align} 
for all $\psi \in C^{\infty} (\mathbb{T}^{2})$. The nonlinear term in \eqref{analytically weak} is finite due to the Calder$\acute{o}$n commutator lemma or its variant (e.g. \cite[Proposition B.1]{CKL21}) as long as $\theta \in L^{2} (0, \infty; \dot{H}^{-\frac{1}{2}} (\mathbb{T}^{2}))$.  

There are three novelties in this work, all of which are related but still interesting individually. 
\begin{enumerate}[label=(\roman*)]
\item The result covers the locally critical/supercritical cases. 
\item The proof consisted of no renormalization constants. 
\item The solution constructed therein possessed the regularity of $L^{2}(0, \infty; \dot{H}^{-\frac{1}{2}}(\mathbb{T}^{2})) \cap C([0, \infty); B_{\infty, 1}^{-\frac{1}{2} - \delta} (\mathbb{T}^{2}))$ for any $\delta > 0$. 
\end{enumerate} 
\noindent Let us elaborate on (i)-(iii) separately. 

First, because the noise is additive and white only in space, it is more natural to deduce the requirement for local subcriticality via rescaling rather than an analogous computation to \eqref{Regularity of STWN}-\eqref{deduction 1}. If $\tilde{\xi}(x) = \lambda \xi(\lambda x)$, then the rescaled solution of \eqref{SQG} is $\tilde{\theta}(t,x) = \lambda^{1+ \alpha - \gamma} \theta(\lambda^{\gamma} t, \lambda x)$ so that they satisfy 
\begin{align*}
\partial_{t} \tilde{\theta} + \lambda^{2\gamma - 2 - \alpha} \mathcal{R}^{\bot} \tilde{\theta} \cdot \nabla \tilde{\theta} + (-\Delta)^{\frac{\gamma}{2}} \tilde{\theta} =  (-\Delta)^{\frac{\alpha}{2}} \tilde{\xi}
\end{align*}
indicating that local subcriticality requires $2 \gamma - 2 - \alpha > 0$. Because \cite[Theorem 1.2]{HZZ22a} holds for all $\alpha \in [0, 1)$ and $\gamma \in [0, \frac{3}{2})$, we observe that this threshold for local criticality can  be violated. 

Such a solution theory, even local-in-time, would be surprising considering Example \ref{Example 2.1}, where we observed that one would not be able to keep track of infinitely many trees and ultimately infinitely many renormalization constants in the locally critical case, let alone locally supercritical case. The explanation is actually part (ii). Hofmanov$\acute{\mathrm{a}}$, Zhu, and Zhu did not have to subtract out any linear equations with ill-defined products; instead, they worked directly on \eqref{SQG}. To be precise, they worked on the equation solved by $f \triangleq \Lambda^{-1} \theta$, but importantly, they worked directly on the SPDE. The noise is white only in space and additive; thus, no renormalization was needed, and only the spatial irregularity had to be dealt with. The crucial ingredient was the convex integration scheme, especially the building blocks, of \cite{CKL21} by Cheng, Kwon, and Li on the stationary solutions (time independent) to the 2D deterministic SQG equation. Using such building blocks, any ill-defined products could be avoided; specifically, for any product of two distributions with compact support that are sufficiently far apart, no resonant terms would arise and hence the product would be well-defined (recall Remark \ref{Remark 2.1}). 

Part (iii) is of surprise particularly because the standard deduction from the fact that $\xi \in \mathscr{C}^{\beta}(\mathbb{T}^{2})$ for $\beta < -1$ would show the regularity of the solution to be 
\begin{equation}\label{solution theta}
\theta \in \mathscr{C}^{\beta}(\mathbb{T}^{2}) \text{ for } \beta < - 1 - \alpha + \gamma \hspace{2mm} \mathbb{P}\text{-a.s.}
\end{equation} 
at best. Considering the special case $\gamma = 0$ and $\alpha = 1- \kappa$ for any $\kappa > 0$, it shows that $\theta \in \mathscr{C}^{\beta}(\mathbb{T}^{2})$ for $\beta < -2 + \kappa$ $\mathbb{P}$-a.s. 
In contrast, the solution constructed in \cite{HZZ22a} via convex integration lies in $C([0, \infty); B_{\infty, 1}^{-\frac{1}{2} - \delta} (\mathbb{T}^{2})) \subset C([0, \infty); \mathscr{C}^{-\frac{1}{2} - \delta} (\mathbb{T}^{2}))$ for any $\delta > 0$. 

In relevance, let us also return to the comment made in Remark \ref{Remark 2.2}. If the solution constructed by the convex integration technique satisfies energy inequality, considering the worst case $\alpha = 1- \kappa$ for some $\kappa > 0$ and $\gamma = 0$ in \eqref{analytically weak}, one may consider  
\begin{align*}
\int_{0}^{t} \int_{\mathbb{T}^{2}} (-\Delta)^{\frac{\alpha}{2}} \xi (-\Delta)^{-\frac{1}{2}} \theta dxds  \lesssim \int_{0}^{t}  \lVert \xi \rVert_{\dot{H}^{\frac{1}{2} - \kappa}} \lVert \theta \rVert_{\dot{H}^{-\frac{1}{2}}} ds 
\end{align*} 
and feel that $\xi$ must lie in $\dot{H}^{\frac{1}{2} - \kappa}(\mathbb{T}^{2})$ $\mathbb{P}$-a.s. Yet, $\xi$ only has the regularity of $\mathscr{C}^{\beta}(\mathbb{T}^{2})$ for $\beta < -1$ $\mathbb{P}$-a.s. As pointed out already in Remark \ref{Remark 2.2}, the misconception here is that the solution constructed by the convex integration technique does not necessarily satisfy the energy inequality.  

The limitation of \cite{HZZ22a} was that the noise could not be rough in time, disallowing the case of the STWN. Independently, the second author in  \cite{Y25b, Y25c}, and together with Walker in \cite{WY24a}, extended the convex integration scheme of \cite{BSV19} by Buckmaster, Shkoller, and Vicol on the deterministic SQG equation to the case of the SQG equation forced by white-in-time noise $\xi$. More recently, Hofmanov$\acute{\mathrm{a}}$, Luo, Zhu, and Zhu \cite{HLZZ23} extended \cite{HZZ22a} to the case $\xi$ is a STWN. A novel new idea therein is to, while working on the equation of $f = \Lambda^{-1} \theta$, define $z^{ 
\scalebox{0.18}{\begin{tikzpicture}
\draw[black, thick] (0.5,1) -- (0.5,0);
\filldraw[black] (0.5,1) circle (6pt); 
\end{tikzpicture}
}}$ to satisfy 
\begin{equation}\label{modified}
\partial_{t} z^{ 
\scalebox{0.18}{\begin{tikzpicture}
\draw[black, thick] (0.5,1) -- (0.5,0);
\filldraw[black] (0.5,1) circle (6pt); 
\end{tikzpicture}
}} + ( - \Delta)^{\frac{\gamma_{1} -1}{2}}z^{ 
\scalebox{0.18}{\begin{tikzpicture}
\draw[black, thick] (0.5,1) -- (0.5,0);
\filldraw[black] (0.5,1) circle (6pt); 
\end{tikzpicture}
}} = (-\Delta)^{\frac{\alpha -1}{2}} \xi
\end{equation} 
for $\gamma_{1} > \gamma$ instead of the conventional 
\begin{equation}\label{conventional}
\partial_{t} z^{ 
\scalebox{0.18}{\begin{tikzpicture}
\draw[black, thick] (0.5,1) -- (0.5,0);
\filldraw[black] (0.5,1) circle (6pt); 
\end{tikzpicture}
}} + ( - \Delta)^{\frac{\gamma -1}{2}} z^{ 
\scalebox{0.18}{\begin{tikzpicture}
\draw[black, thick] (0.5,1) -- (0.5,0);
\filldraw[black] (0.5,1) circle (6pt); 
\end{tikzpicture}
}} = (-\Delta)^{\frac{\alpha -1}{2}} \xi
\end{equation} 
according to \eqref{Define z1}. Informally, employing \eqref{modified} instead of \eqref{conventional} tames the temporal irregularity arising from the STWN $\xi$. On the other hand, as expected from our discussions in Example \ref{Example 2.1}, due to the introduction of $z^{ 
\scalebox{0.18}{\begin{tikzpicture}
\draw[black, thick] (0.5,1) -- (0.5,0);
\filldraw[black] (0.5,1) circle (6pt); 
\end{tikzpicture}
}}$, the authors of \cite{HLZZ23} had to work with a renormalization constant; we refer to \cite{HZZ22a, HLZZ23} for further details.
\end{example} 

\section{Heat equation with damping of odd power}
The purpose of this section is to give a different perspective concerning applications of the convex integration technique to singular SPDEs. As we described in Section \ref{Section 2.1}, it is desired that the convex integration technique can be applied to other equations of mathematical physics, not just those of fluid mechanics; potential examples include the KPZ equation \eqref{Define KPZ}, the $\Phi^{4}$ model \eqref{Define Phi4}, and the Yang-Mills equation \eqref{Yang-Mills}. We focus on the $\Phi^{4}$ model \eqref{Define Phi4}, and explain our reasonings as follows. Every successful application of the convex integration technique on SPDE thus far was preceded by the deterministic case. Hence, we zoom in on 
\begin{equation}\label{cubic heat}
\partial_{t} u + u^{3} = \Delta u, \hspace{3mm} x \in \mathbb{T}^{d},  
\end{equation} 
which is a heat equation with cubic damping. While this suggests that non-uniqueness is highly unlikely, we recall the interesting example of Tychonov's solution to heat equation (see \cite{S23, Y16} and \cite[Section 7]{J82}): for $x \in \mathbb{R}$ and $\alpha > 1$,  
\begin{equation}\label{Tychonov}
u(t,x) = \sum_{k=0}^{\infty} \frac{ \partial_{t}^{k} g(t)}{(2k)!} \lvert x \rvert^{2k} \hspace{1mm} \text{ where } \hspace{1mm}  g(t) \triangleq 
\begin{cases}
e^{- \frac{1}{t^{\alpha}}} & \text{ for } t > 0, \\
0 & \text{ for } t \leq 0, 
\end{cases}
\end{equation} 
that satisfies 
\begin{align*}
\partial_{t} u = \Delta u, \hspace{3mm} u\rvert_{t=0} \equiv 0. 
\end{align*}
The solution $u$ from \eqref{Tychonov} verifies non-uniqueness in comparison to a zero solution, in fact, infinitely many solutions by considering different values of $\alpha > 1$. 
\begin{remark}\label{Remark on the sign}
We also recall that when the cubic nonlinearity does not have damping effect, i.e. when $u^{3}$ is on the right-hand side of \eqref{cubic heat}, there is abundance of non-uniqueness results in the literature; e.g. \cite[Theorem 3]{NS85} by Ni and Sacks and \cite[Corollary 0.9]{T02} by Terraneo in $C([0, T]; L^{3}(D))$ with $D$ being a unit ball and $\mathbb{R}^{d}$, respectively. This is in sharp contrast to the Navier-Stokes equations, or almost every other PDEs in fluid mechanics, where the sign of the nonlinear term can be readily intertwined by replacing the solution $u$ by $-u$. More precisely, if $u$ satisfies 
\begin{equation*}
\partial_{t} u  + \nabla \pi =  \divergence (u\otimes u) +  \Delta u, \hspace{3mm} \nabla\cdot u = 0, 
\end{equation*} 
then $-u$ solves the deterministic Navier-Stokes equations, \eqref{Navier-Stokes} with zero noise. In contrast, even if $u$ solves \eqref{cubic heat}, $-u$ would not satisfy 
\begin{align*}
\partial_{t} u = u^{3} + \Delta u.
\end{align*}
\end{remark}

Now, for generality we consider
\begin{equation}\label{odd power}
\partial_{t} u + u^{n} = \Delta u, \hspace{3mm} u \rvert_{t=0} = u^{\text{in}} \in L^{n} (\mathbb{T}^{d}), 
\end{equation}
where $n \in \mathbb{N}$ is any odd number such that $n \geq 3$. We state its definition precisely following Definition \ref{Definition 2.1} but with minimum regularity needed for its weak formulation: 
\begin{define}\label{Definition 4.1}
Given initial data $u^{\text{in}} \in L^{n} (\mathbb{T}^{d})$, a scalar field $u \in L^{n} (0, \infty; L^{n} (\mathbb{T}^{d}))$ is a weak solution to \eqref{odd power} if 
\begin{align}
&\int_{\mathbb{T}^{d}} u(t,x) \psi(t,x) dx - \int_{\mathbb{T}^{d}} u^{\text{in}} (x) \psi(0,x) dx - \int_{0}^{t} \int_{\mathbb{T}^{d}} u(s,x) \partial_{s} \psi(s,x) dx ds  \nonumber \\
& \hspace{30mm} + \int_{0}^{t} \int_{\mathbb{T}^{d}} u^{n}(s,x) \psi(s,x) dx ds = \int_{0}^{t} \int_{\mathbb{T}^{d}} u(s,x) \Delta \psi(s,x) dx ds \label{weak formulation}
\end{align} 
for all test function $\psi$ in $C_{c}^{\infty} ( (-\infty, T) \times \mathbb{T}^{d})$. 
\end{define} 

Our point of view is the following. If the technique of convex integration is ever applied successfully on \eqref{odd power}, then its weak solution will possess the regularity of $u \in L^{n} (0, \infty; L^{n} (\mathbb{T}^{d}))$ at minimum so that the nonlinear term becomes well-defined. To the best of our knowledge, every convex integration solution resulted in non-uniqueness, which is natural considering its origin in relation to homotopy principle. Therefore, by proving the uniqueness of such a weak solution with minimum regularity of $L^{n} (0, \infty; L^{n} (\mathbb{T}^{d}))$ in the following Theorem \ref{Theorem 4.1}, of which we are not aware in the literature to the best of our knowledge, we can rule out the possibility of any solution to \eqref{odd power} constructed via convex integration, at least within the scope of implementations of the existing method.

\begin{theorem}\label{Theorem 4.1} 
Let $n \in \mathbb{N}$ be any odd number such that $n \geq 3$. Suppose that $u \in L^{n} (0,\infty; L^{n}(\mathbb{T}^{d}))$ is a weak solution to \eqref{odd power} according to Definition \ref{Definition 4.1}. Then $u$ is unique. 
\end{theorem}

The proof of Theorem \ref{Theorem 4.1} utilizes Steklov average, of which we collect its definition and useful properties in the following Lemma \ref{Lemma on Steklov} for readers' convenience. 
\begin{lemma}\label{Lemma on Steklov}
Let $I \subset \mathbb{R}$ be any interval, $E$ any measurable set over $\mathbb{T}^{d}$, $1 \leq q, r \leq \infty$ and $h > 0$. Given $v \in L^{r} (I;  L^{q}(E))$, we define a Steklov average by 
\begin{align*}
v_{h} (t, \cdot) \triangleq \frac{1}{h} \int_{t-h}^{t} \tilde{v} (s, \cdot) ds \hspace{2mm} \text{ where } \hspace{2mm}  \tilde{v}(t, \cdot) \triangleq 
\begin{cases}
v(t, \cdot) &\text{ if } t \in I, \\
0 & \text{ if } t \in \mathbb{R} \setminus I.
\end{cases}
\end{align*}
Then 
\begin{enumerate}
\item (e.g. \cite[p. 9]{BD25}) $\lVert v_{h} \rVert_{L^{q}(I; L^{q}(E))} \leq \lVert v \rVert_{L^{q}(I; L^{q}(E))}$ and $(\partial_{t} v_{h})(t,x) = \frac{\tilde{v}(t,x) - \tilde{v}(t-h,x)}{h}$. 
\item (e.g. \cite[Lemma 2.4]{CDG17}) $v_{h} \in C(I; L^{q}(E)) \cap L^{\infty} (I; L^{q}(E))$. 
\item (e.g. \cite[Lemma 2.5]{CDG17}) $v_{h} \to v$ in $L^{r} (I; L^{q}(E))$ as $h\searrow 0$. 
\end{enumerate} 
\end{lemma} 

\begin{proof}[Proof of Theorem \ref{Theorem 4.1}]
To mollify in space, we consider $\rho \in C_{c}^{\infty} (\mathbb{R}^{d})$ that is non-negative such that $\int_{\mathbb{R}^{d}} \rho(x) dx = 1$, define a mollifier $\rho_{\epsilon}(x) \triangleq \epsilon^{-d} \rho(\frac{x}{\epsilon})$, periodize in $x \in \mathbb{T}^{d}$ in the standard way, and relabel it as $\rho$ so that for any $f \in L^{1} (\mathbb{T}^{d})$, we can define $f_{\epsilon}(x) \triangleq (f \ast_{x} \rho_{\epsilon})(x)$. 

Now, similarly to \cite{BD25}, we fix $T> 0$ and assume that given $u^{\text{in}} \in L^{n} (\mathbb{T}^{d})$, $u_{j}$, $j \in \{1,2\}$, are both weak solutions from same initial data according to Definition \ref{Definition 4.1} so that they satisfy 
\begin{align*}
&- \int_{\mathbb{T}^{d}} u^{\text{in}} (x) \psi(0,x) dx - \int_{0}^{T} \int_{\mathbb{T}^{d}} u(s,x) \partial_{s} \psi(s,x) dx ds  \nonumber \\
&  = - \int_{0}^{T} \int_{\mathbb{T}^{d}} u^{n}(s,x) \psi(s,x) dx ds + \int_{0}^{T} \int_{\mathbb{T}^{d}} u(s,x) \Delta \psi(s,x) dx ds
\end{align*} 
since $\psi(T) \equiv 0$. We extend $u_{j} \equiv 0$ on $(-\infty, 0)$. Then, denoting by $u_{h}(t,x) \triangleq \frac{1}{h} \int_{t-h}^{t} u(s,x) ds$, we work on 
\begin{align*}
& - \int_{\mathbb{T}^{d}} u^{\text{in}} (x) \psi(0,x) dx - \int_{0}^{T} \int_{\mathbb{T}^{d}} u_{j,h} (s,x) \partial_{s} \psi(s,x) dx ds \\
& =- \int_{0}^{T} \int_{\mathbb{T}^{d}} u_{j,h}^{n} (s,x) \psi(s,x) dx ds + \int_{0}^{T} \int_{\mathbb{T}^{d}} u_{j,h} (s,x) \Delta \psi(s,x) dx ds. 
\end{align*} 
We integrate by parts in the time derivative to obtain 
\begin{align*}
& - \int_{\mathbb{T}^{d}} u^{\text{in}} (x) \psi(0,x) dx + \int_{0}^{T} \int_{\mathbb{T}^{d}} \partial_{s}  u_{j,h} (s,x) \psi(s,x) dx ds \\
& = - \int_{0}^{T} \int_{\mathbb{T}^{d}} u_{j,h}^{n} (s,x) \psi(s,x) dx ds + \int_{0}^{T} \int_{\mathbb{T}^{d}} u_{j,h} (s,x) \Delta \psi(s,x) dx ds
\end{align*}
as $\psi(T) \equiv 0$ and $u_{j,h}(0,x) = 0$ for both $j\in \{1,2\}$. Thus, for all $v \in C_{0}^{\infty} ( \mathbb{R} \times C^{\infty} (\mathbb{T}^{d}))$ and $t \in (0, T)$, we have 
\begin{align*}
& - \int_{\mathbb{T}^{d}} u^{\text{in}} (x) v(0,x) dx + \int_{0}^{t} \int_{\mathbb{T}^{d}} \partial_{s}u_{j,h} (s,x) v(s,x) dx ds \\
& = - \int_{0}^{t} \int_{\mathbb{T}^{d}} u_{j,h}^{n} (s,x) v(s,x) dx ds + \int_{0}^{t} \int_{\mathbb{T}^{d}} u_{j,h} (s,x) \Delta v(s,x) dx ds. 
\end{align*}
We define $w \triangleq u_{1} - u_{2}$ so that 
\begin{align*}
u_{1}^{n} - u_{2}^{n} = (u_{1} - u_{2}) [ u_{1}^{n-1} + u_{1}^{n-2} u_{2} + \hdots + u_{2}^{n-1} ] = w \sum_{l=0}^{n-1} u_{1}^{n-1-l} u_{2}^{l} 
\end{align*}
to deduce 
\begin{align}
&\int_{0}^{t} \int_{\mathbb{T}^{d}} \partial_{s} w_{h}(s,x) v(s,x) dx ds \nonumber \\
&= - \int_{0}^{t} \int_{\mathbb{T}^{d}} [ w_{h} \sum_{l=0}^{n-1} u_{1,h}^{n-1-l} u_{2,h}^{l} ] v(s,x) dx ds + \int_{0}^{t} \int_{\mathbb{T}^{d}} w_{h}(s,x) \Delta v(s,x) dx ds.   \label{NEW2}
\end{align}
Now $\left( \frac{w_{h,\epsilon}}{\sqrt{ \delta + w_{h,\epsilon}^{2}}}\right)_{\epsilon}$ is $C^{\infty}$ in space and by Lemma \ref{Lemma on Steklov} (1), for all $h > 0$ fixed, $\partial_{t} w_{h} \in L_{T,x}^{n}$ so that by approximation we can take $v = \left( \frac{w_{h,\epsilon}}{\sqrt{ \delta + w_{h,\epsilon}^{2}}}\right)_{\epsilon}$ in \eqref{NEW2} to obtain
\begin{align*}
&\int_{0}^{t} \int_{\mathbb{T}^{d}} \partial_{s} w_{h}(s,x) \left( \frac{w_{h,\epsilon}}{\sqrt{ \delta + w_{h,\epsilon}^{2}}}\right)_{\epsilon} (s,x) dx ds \\
&= - \int_{0}^{t} \int_{\mathbb{T}^{d}} [ w_{h} \sum_{l=0}^{n-1} u_{1,h}^{n-1-l} u_{2,h}^{l} ] \left( \frac{w_{h,\epsilon}}{\sqrt{ \delta + w_{h,\epsilon}^{2}}}\right)_{\epsilon} (s,x) dx ds \\
& \hspace{10mm}  + \int_{0}^{t} \int_{\mathbb{T}^{d}} w_{h}(s,x) \Delta \left( \frac{w_{h,\epsilon}}{\sqrt{ \delta + w_{h,\epsilon}^{2}}}\right)_{\epsilon} (s,x) dx ds.  
\end{align*}
We then shift the mollification to deduce 
\begin{align*}
&\int_{0}^{t} \int_{\mathbb{T}^{d}} \partial_{s} w_{h,\epsilon}(s,x) \left( \frac{w_{h,\epsilon}}{\sqrt{ \delta + w_{h,\epsilon}^{2}}}\right) (s,x) dx ds \\
&= - \int_{0}^{t} \int_{\mathbb{T}^{d}} [ w_{h} \sum_{l=0}^{n-1} u_{1,h}^{n-1-l} u_{2,h}^{l} ]_{\epsilon} \left( \frac{w_{h,\epsilon}}{\sqrt{ \delta + w_{h,\epsilon}^{2}}}\right) (s,x) dx ds \\
& \hspace{10mm}  + \int_{0}^{t} \int_{\mathbb{T}^{d}} w_{h,\epsilon}(s,x) \Delta \left( \frac{w_{h,\epsilon}}{\sqrt{ \delta + w_{h,\epsilon}^{2}}}\right) (s,x) dx ds.  
\end{align*}
We integrate by parts on the first one with respect to time to deduce 
\begin{align*}
&\int_{\mathbb{T}^{d}} w_{h,\epsilon}(t,x) \left( \frac{w_{h,\epsilon}}{\sqrt{ \delta + w_{h,\epsilon}^{2}}} \right)(t,x)dx- \int_{0}^{t} \int_{\mathbb{T}^{d}} w_{h,\epsilon}(s,x) \left(\frac{ \partial_{s} w_{h,\epsilon}\delta}{(\delta + w_{h,\epsilon}^{2})^{\frac{3}{2}}}\right)(s,x) dxds\\
&= - \int_{0}^{t} \int_{\mathbb{T}^{d}} [ w_{h} \sum_{l=0}^{n-1} u_{1,h}^{n-1-l} u_{2,h}^{l} ]_{\epsilon} \left( \frac{w_{h,\epsilon}}{\sqrt{ \delta + w_{h,\epsilon}^{2}}}\right) (s,x) dx ds  \\
& \hspace{10mm} + \int_{0}^{t} \int_{\mathbb{T}^{d}} w_{h,\epsilon}(s,x) \Delta \left( \frac{w_{h,\epsilon}}{\sqrt{ \delta + w_{h,\epsilon}^{2}}}\right) (s,x) dx ds  
\end{align*} 
where we used the fact that $w_{h,\epsilon}(0) \equiv 0$. We can compute the diffusion term as 
\begin{align*}
&\int_{0}^{t} \int_{\mathbb{T}^{d}} w_{h,\epsilon}(s,x) \Delta \left( \frac{w_{h,\epsilon}}{\sqrt{ \delta + w_{h,\epsilon}^{2}}}\right) (s,x) dx ds \\
& \hspace{10mm} = - \int_{0}^{t} \int_{\mathbb{T}^{d}} \lvert \nabla w_{h,\epsilon} (s,x) \rvert^{2}  \frac{\delta}{(\delta + w_{h,\epsilon}^{2})^{\frac{3}{2}}} dx ds\leq 0
\end{align*} 
so that 
\begin{align}
&\int_{\mathbb{T}^{d}} w_{h,\epsilon}(t,x) \left( \frac{w_{h,\epsilon}}{\sqrt{ \delta + w_{h,\epsilon}^{2}}} \right)(t,x) dx- \int_{0}^{t} \int_{\mathbb{T}^{d}} w_{h,\epsilon}(s,x)   \left(\frac{\partial_{s} w_{h,\epsilon} \delta}{(\delta + w_{h,\epsilon}^{2})^{\frac{3}{2}}} \right)(s,x)dxds  \nonumber \\
& \hspace{10mm} \leq - \int_{0}^{t} \int_{\mathbb{T}^{d}} [ w_{h} \sum_{l=0}^{n-1} u_{1,h}^{n-1-l} u_{2,h}^{l} ]_{\epsilon} \left( \frac{w_{h,\epsilon}}{\sqrt{ \delta + w_{h,\epsilon}^{2}}}\right) (s,x) dx ds.  \label{NEW1}
\end{align} 

To handle the nonlinear term, we split it to three parts: 
\begin{align}\label{nonlinear} 
-\int_{0}^{t} \int_{\mathbb{T}^{d}} [ w_{h} \sum_{l=0}^{n-1} u_{1,h}^{n-1-l} u_{2,h}^{l} ]_{\epsilon} \left(\frac{w_{h,\epsilon}}{ \sqrt{ \delta +  w_{\epsilon}^{2}}}\right) dx ds = I_{1,h, \epsilon} + I_{2,h, \epsilon} + I_{3,h, \epsilon}, 
\end{align} 
where 
\begin{subequations}
\begin{align}
I_{1,h,\epsilon} =& - \int_{0}^{t} \int_{\mathbb{T}^{d}} \Bigg( \left[ w_{h} \sum_{l=0}^{n-1} u_{1,h}^{n-1-l} u_{2,h}^{l} \right]_{\epsilon} -w_{h} \sum_{l=0}^{n-1} u_{1,h}^{n-1-l} u_{2,h}^{l} \Bigg) \frac{ w_{h,\epsilon}}{\sqrt{\delta + w_{h,\epsilon}^{2}}} dx ds, \\
I_{2,h,\epsilon} =&  - \int_{0}^{t} \int_{\mathbb{T}^{d}} (w_{h}- w_{h,\epsilon}) \sum_{l=0}^{n-1} u_{1,h}^{n-1-l} u_{2,h}^{l}  \frac{w_{h,\epsilon}}{\sqrt{\delta + w_{h,\epsilon}^{2}}} dx ds, \label{Define I2}\\
I_{3,h,\epsilon} =& -\int_{0}^{t} \int_{\mathbb{T}^{d}} w_{h,\epsilon} \sum_{l=0}^{n-1} u_{1,h}^{n-1-l} u_{2,h}^{l} \frac{w_{h,\epsilon}}{\sqrt{\delta + w_{h,\epsilon}^{2}}} dx ds. \label{Define I3}
\end{align} 
\end{subequations}  
First, we bound 
\begin{align}\label{Bound on I1}
I_{1,h,\epsilon}\leq \Bigg\lVert  \lVert  \left[w_{h} \sum_{l=0}^{n-1} u_{1,h}^{n-1-l} u_{2,h}^{l} \right]_{\epsilon} - w_{h} \sum_{l=0}^{n-1} u_{1,h}^{n-1-l} u_{2,h}^{l}  \rVert_{L^{1}(\mathbb{T}^{d})} \Bigg\rVert_{L^{1}(0,t)}. 
\end{align}
We have for almost (a.e.) $s \in [0,t]$, $w_{h}(s) \sum_{l=0}^{n-1} u_{1,h}^{n-1-l} (s)u_{2,h}^{l}(s) \in L^{1}(\mathbb{T}^{d})$ because $w, u_{1}, u_{2} \in L^{n} (0,\infty; L^{n} (\mathbb{T}^{d}))$ by hypothesis and thus by Lemma \ref{Lemma on Steklov} (2), $w_{h}, u_{1,h}, u_{2,h} \in L^{\infty} (0, \infty; L^{n}(\mathbb{T}^{d}))$. Thus, the standard result on mollifiers allows us to take $\epsilon \searrow 0$ to deduce 
\begin{equation}\label{est 1}
\lim_{\epsilon \searrow 0}  I_{1,h,\epsilon} \leq 0. 
\end{equation} 

Next, we can estimate $I_{2,h,\epsilon}$ in \eqref{Define I2} by 
\begin{align}\label{est 2} 
I_{2,h,\epsilon}  \leq \lVert w_{h} - w_{h,\epsilon} \rVert_{L^{n} (0, t; L^{n} (\mathbb{T}^{d})) } \sum_{l=0}^{n-1} \lVert u_{1,h} \rVert_{L^{n} (0, t; L^{n} (\mathbb{T}^{d}))}^{n-1-l} \lVert u_{2,h} \rVert_{L^{n} (0, t; L^{n} (\mathbb{T}^{d}))}^{l} 
\end{align}
and therefore 
\begin{equation}
\lim_{\epsilon \searrow 0}  I_{2,h,\epsilon} \leq 0. 
\end{equation} 
Finally, for $I_{3,h,\epsilon}$ in \eqref{Define I3} we first write in details 
\begin{align}
&- w_{h,\epsilon} \sum_{l=0}^{n-1} u_{1,h}^{n-1-l} u_{2,h}^{l} \frac{w_{h,\epsilon}}{\sqrt{\delta + w_{h,\epsilon}^{2}}}  \nonumber\\
=& - \frac{w_{h,\epsilon}^{2}}{\sqrt{\delta + w_{h,\epsilon}^{2}}}   [ u_{1,h}^{n-1} + u_{1,h}^{n-3} u_{2,h}^{2} + \hdots + u_{1,h}^{2} u_{2,h}^{n-3} + u_{2,h}^{n-1} ]  \nonumber \\
&  -  \frac{w_{h,\epsilon}^{2}}{\sqrt{\delta + w_{h,\epsilon}^{2}}} [ u_{1,h}^{n-2} u_{2,h} + u_{1,h}^{n-4}  u_{2,h}^{3} + \hdots + u_{1,h}^{3} u_{2,h}^{n-4} + u_{1,h} u_{2,h}^{n-2} ]. \label{inductive} 
\end{align}
Considering that $n\in \mathbb{N}, n \geq 3$ is odd, the first bracket is already non-positive, leaving our task to the second bracket. In fact, we will show that this sum is non-positive as a whole only with the help from the first bracket. We can estimate 
\begin{align*} 
& - \frac{w_{h,\epsilon}^{2}}{\sqrt{\delta + w_{h,\epsilon}^{2}}} u_{1,h}^{n-2} u_{2,h}  \leq \frac{w_{h,\epsilon}^{2}}{\sqrt{\delta + w_{h,\epsilon}^{2}}} \frac{1}{2} [ u_{1,h}^{n-1} + u_{1,h}^{n-3} u_{2,h}^{2}], \\
& - \frac{w_{h,\epsilon}^{2}}{\sqrt{\delta + w_{h,\epsilon}^{2}}} u_{1,h}^{n-4} u_{2,h}^{3}   \leq \frac{w_{h,\epsilon}^{2}}{\sqrt{\delta + w_{h,\epsilon}^{2}}} \frac{1}{2} [ u_{1,h}^{n-3} u_{2,h}^{2}+  u_{1,h}^{n-5} u_{2,h}^{4} ], \\
& - \frac{w_{h,\epsilon}^{2}}{\sqrt{\delta + w_{h,\epsilon}^{2}}} u_{1,h}^{n-6} u_{2,h}^{5}   \leq \frac{w_{h,\epsilon}^{2}}{\sqrt{\delta + w_{h,\epsilon}^{2}}} \frac{1}{2} [ u_{1,h}^{n-5} u_{2,h}^{4}+  u_{1,h}^{n-7} u_{2,h}^{6} ], \\
&\vdots \\
& - \frac{w_{h,\epsilon}^{2}}{\sqrt{\delta + w_{h,\epsilon}^{2}}} u_{1,h}^{5} u_{2,h}^{n-6}   \leq \frac{w_{h,\epsilon}^{2}}{\sqrt{\delta + w_{h,\epsilon}^{2}}} \frac{1}{2} [ u_{1,h}^{4} u_{2,h}^{n-5} + u_{1,h}^{6} u_{2,h}^{n-7} ], \\
& - \frac{w_{h,\epsilon}^{2}}{\sqrt{\delta + w_{h,\epsilon}^{2}}} u_{1,h}^{3} u_{2,h}^{n-4}   \leq \frac{w_{h,\epsilon}^{2}}{\sqrt{\delta + w_{h,\epsilon}^{2}}} \frac{1}{2} [ u_{1,h}^{2} u_{2,h}^{n-3} + u_{1,h}^{4} u_{2,h}^{n-5} ], \\
& - \frac{w_{h,\epsilon}^{2}}{\sqrt{\delta + w_{h,\epsilon}^{2}}} u_{1,h} u_{2,h}^{n-2}   \leq \frac{w_{h,\epsilon}^{2}}{\sqrt{\delta + w_{h,\epsilon}^{2}}} \frac{1}{2} [ u_{2,h}^{n-1} + u_{1,h}^{2} u_{2,h}^{n-3} ]. 
\end{align*}
Thus, 
\begin{equation}
- w_{h,\epsilon} \sum_{l=0}^{n-1} u_{1,h}^{n-1-l} u_{2,h}^{l} \frac{w_{h,\epsilon}}{\sqrt{\delta + w_{h,\epsilon}^{2}}} \leq - \frac{1}{2} \frac{w_{h,\epsilon}^{2}}{\sqrt{\delta + w_{h,\epsilon}^{2}}} [ u_{1,h}^{n-1} + u_{2,h}^{n-1} ] \leq 0. 
\end{equation} 
Applying this result to \eqref{Define I3} allows us to conclude that 
\begin{equation}\label{est 3}
I_{3,h,\epsilon} \leq 0. 
\end{equation} 
Applying \eqref{est 1}, \eqref{est 2}, and \eqref{est 3} to \eqref{nonlinear} allows us to conclude that 
\begin{equation}\label{est 4}
\lim_{\epsilon \to 0} \left( -\int_{0}^{t} \int_{\mathbb{T}^{d}} \left[ w_{h} \sum_{l=0}^{n-1} u_{1,h}^{n-1-l} u_{2,h}^{l} \right]_{\epsilon} \frac{w_{h,\epsilon}}{\sqrt{\delta + w_{h,\epsilon}^{2}}} dx ds \right) \leq 0. 
\end{equation} 
Thus, taking $\epsilon \searrow 0$ in \eqref{NEW1} and applying \eqref{est 4} gives us 
\begin{align*}
 \int_{\mathbb{T}^{d}} w_{h} (t,x) \left( \frac{w_{h}}{\sqrt{ \delta + w_{h}^{2}}} \right) (t,x) dx  - \int_{0}^{t} \int_{\mathbb{T}^{d}} w_{h}(s,x)    \left(\frac{\partial_{s} w_{h}\delta}{( \delta + w_{h}^{2})^{\frac{3}{2}}}\right)(s,x)  dxds  \leq 0.
\end{align*}
Taking $\delta \searrow 0$ gives us $\lVert w_{h}(t) \rVert_{L^{1}(\mathbb{T}^{d})} = 0$ so that $\lVert w_{h} \rVert_{L^{1}((0, T) \times \mathbb{T}^{d})} = 0$. Taking $h\searrow 0$ gives us $\lVert w \rVert_{L^{1}((0, T) \times \mathbb{T}^{d})} = 0$ according to Lemma \ref{Lemma on Steklov} (3). At last, this implies $\lVert w(t) \rVert_{L^{1}(\mathbb{T}^{d})} = 0$ for almost every $t$, allowing us to conclude our proof.
\end{proof} 

Theorem \ref{Theorem 4.1} indicates that, unless the STWN $\xi$ induces non-uniqueness, a successful application of the convex integration technique on the $\Phi^{4}$ model \eqref{Define Phi4}, which was suggested as an open problem in \cite[Section 4]{Y25e}, seems highly unlikely. To the best of the authors' knowledge, there has not been a case in which convex integration technique has been achieved on an SPDE and its proof fails in the case of zero noise. in fact, in all the successful accounts of convex integration on SPDEs thus far, convex integration was always achieved in the deterministic case first and thereafter extended in the presence of random force. 

On the other hand, although the Yang-Mills equation \eqref{Yang-Mills} consists of nonlinearity of a cubic type (besides the Burgers' type), we believe that nothing definitive can be stated at the moment concerning its potential application of the convex integration technique due to its complex structure. The proof of Theorem \ref{Theorem 4.1} crucially depended on the sign of the nonlinearity (recall Remark \ref{Remark on the sign}) and the cubic terms in \eqref{Yang-Mills} are not necessarily damping. 

\section{Conclusion} 
We reviewed recent developments of convex integration technique applied on singular SPDEs. As we described through Example \ref{Example on singular NS} (also Remark \ref{Remark 3.3}), the technique has the potential to extend a known local solution to become global-in-time, although non-unique, when no other approach seems available. In fact, as we described through Example \ref{Example on singular SQG}, it can produce a global-in-time solution when not even a local-in-time solution is known to exist. Finally, in Theorem \ref{Theorem 4.1}  we showed that a weak solution to the deterministic heat equation with damping of odd power in any dimension that merely has minimum regularity needed for the nonlinear term to be well-defined is already unique, making the case that the application of the convex integration technique on the $\Phi^{4}$ model or any variant replacing $u^{3}$ by $u^{n}$ for $n\in\mathbb{N}$ with $n \geq 3$ odd, is highly unlikely unless the STWN induces non-uniqueness. 

\section*{Acknowledgments}
We express deep gratitude to the editor and the anonymous reviewer for the careful reading and valuable suggestions and comments that have improved this manuscript significantly. The second author expresses deep gratitude to Prof. Alexey Cheskidov and Prof. Matthew Novack for valuable discussions. 

\bibliographystyle{amsalpha}

\end{document}